\documentclass[11pt]{amsart}

\usepackage[numbers,sort]{natbib}

\usepackage{amsfonts}
\usepackage{amsmath}
\usepackage{amssymb}
\usepackage{amscd}
\usepackage{color}
\usepackage{xcolor}
\usepackage{amsthm}
\usepackage[hmargin=3.9cm,vmargin=3.895cm]{geometry}

\usepackage{mathrsfs}

\usepackage{enumitem}

\usepackage[hyperindex]{hyperref}

\usepackage{etoolbox}
\makeatletter
\let\ams@starttoc\@starttoc
\makeatother
\usepackage{parskip}
\makeatletter
\let\@starttoc\ams@starttoc
\patchcmd{\@starttoc}{\makeatletter}{\makeatletter\parskip\z@}{}{}
\makeatother

\newtheorem{thm}{Theorem}

\newtheorem{thm*}{Theorem}
\newtheorem{prop}{Proposition}
\newtheorem{lma}[prop]{Lemma}

\newtheorem{cor}[prop]{Corollary}

\theoremstyle{definition}

\newtheorem{df}[prop]{Definition} 

\newtheorem{conj}[prop]{Conjecture}

\theoremstyle{remark}

\newtheorem{rmk}[prop]{Remark} 

\newcommand{\F}{{\mathbb{F}}}
\newcommand{\R}{{\mathbb{R}}}
\newcommand{\Z}{{\mathbb{Z}}}
\newcommand{\C}{{\mathbb{C}}}
\newcommand{\Q}{{\mathbb{Q}}}

\newcommand{\bK}{{\mathbb{K}}}
\newcommand{\bF}{{\mathbb{F}}}

\newcommand{\mf}[1]{{\mathscr{#1}}}

\newcommand{\ra}{\rightarrow}
\newcommand{\del}{\partial}

\newcommand{\sm}[1]{C^\infty(#1)}

\newcommand{\A}{\mathcal{A}}

\newcommand{\cL}{\mathcal{L}}

\newcommand{\intoi}{\int_0^1}
\newcommand{\til}[1]{\widetilde{#1}}

\newcommand{\ol}[1]{\overline{#1}}

\newcommand{\zt}{{\Z/(2)}}

\DeclareMathOperator{\ima}{\mathrm{im}}

\newcommand{\om}{\omega}
\newcommand{\al}{\alpha}
\newcommand{\la}{\lambda}

\newcommand{\ga}{\gamma}
\newcommand{\eps}{\epsilon}

\newcommand{\cA}{\mathcal{A}}
\newcommand{\cB}{\mathcal{B}}

\newcommand{\cF}{\mathcal{F}}

\newcommand{\cH}{\mathcal{H}}

\newcommand{\cJ}{\mathcal{J}}

\newcommand{\cO}{\mathcal{O}}

\newcommand{\cP}{\mathcal{P}}

\newcommand{\fF}{\mf{F}}

\newcommand{\rT}{\mathrm{T}}
\newcommand{\rS}{\mathrm{S}}
\newcommand{\fix}{\mathrm{Fix}}

\newcommand{\pr}{pseudo-rotation}

\def\mrm#1{{\mathrm{#1}}}

\def\cl#1{{\mathcal{#1}}}

\newcommand{\brat}[1]{{\left< #1 \right>}}

\DeclareMathOperator{\Ham}{\mathrm{Ham}}

\DeclareMathOperator{\im}{\mathrm{im}}

\DeclareMathOperator{\id}{\mathrm{id}}
\DeclareMathOperator{\spec}{\mathrm{Spec}}
\DeclareMathOperator{\Spec}{\mathrm{Spec}}
\DeclareMathOperator{\loc}{\mathrm{loc}}

\DeclareMathOperator{\diam}{\mathrm{diam}}

\def\H2{H^{(2)}}

\newcommand{\esemail}{shelukhin@dms.umontreal.ca}
\newcommand{\msaemail}{marcelo.sarkis.atallah@umontreal.ca}

\begin{document}

\title{Hamiltonian no-torsion}

\author{Marcelo S. Atallah}
\address{Marcelo S. Atallah, Department of Mathematics and Statistics,
	University of Montreal, C.P. 6128 Succ.  Centre-Ville Montreal, QC
	H3C 3J7, Canada}
\email{\msaemail}

	\author{Egor Shelukhin}
\address{Egor Shelukhin, Department of Mathematics and Statistics,
	University of Montreal, C.P. 6128 Succ.  Centre-Ville Montreal, QC
	H3C 3J7, Canada}
\email{\esemail}

\begin{abstract}
In 2002 Polterovich has notably established that on closed aspherical symplectic manifolds, Hamiltonian diffeomorphisms of finite order, which we call Hamiltonian torsion, must in fact be trivial. In this paper we prove the first higher-dimensional Hamiltonian no-torsion theorems beyond the symplectically aspherical case. We start by showing that closed symplectic Calabi-Yau and negative monotone symplectic manifolds do not admit Hamiltonian torsion. Going still beyond topological constraints, we prove that every closed positive monotone symplectic manifold $(M,\om)$ admitting Hamiltonian torsion is geometrically uniruled by holomorphic spheres for every $\om$-compatible almost complex structure, partially answering a question of McDuff-Salamon. This provides many additional no-torsion results, and as a corollary yields the geometric uniruledness of monotone Hamiltonian $S^1$-manifolds, a fact closely related to a celebrated result of McDuff from 2009. Moreover, the non-existence of Hamiltonian torsion implies the triviality of Hamiltonian actions of lattices like $SL(k,\Z)$ for $k \geq 2,$ as well as those of compact Lie groups. Finally, for monotone symplectic manifolds admitting Hamiltonian torsion, we prove an analogue of Newman's theorem on finite transformation groups for several natural norms on the Hamiltonian group: such subgroups cannot be contained in arbitrarily small neighborhoods of the identity. Our arguments rely on generalized Morse-Bott methods, as well as on quantum Steenrod powers and Smith theory in filtered Floer homology.
\end{abstract}

\maketitle

\tableofcontents

\section{Introduction and main results}

\subsection{Introduction}

The question of the existence of finite group actions on manifolds has been of interest in topology for a long time. In particular, it is in order to study this question that P. A. Smith \cite{Smith-original} has developed in the $1930$s what is now called Smith theory for cohomology with $\F_p$ coefficients in the context of continuous actions of finite $p$-groups. We refer the reader to  \cite{Floyd-original,Bredon-Transformation,Borel-Transformation,Hsiang-Transformation} for further references on Smith theory. 

Quite a lot of progress regarding this question has been obtained in low-dimensional topology \cite{MorganBass-SmithConj, ChenKwasik-K3} and in smooth topology in arbitrary dimension - see for example \cite{Mundet-JTop} and references therein. As a first easy example, we remark that it is not hard to classify finite group actions on closed surfaces. 

In symplectic topology, it was shown by Polterovich \cite{Pol-notors} that non-trivial {\em Hamiltonian} finite group actions, that we refer to as {\em Hamiltonian torsion}, on symplectically aspherical manifolds do not exist. Furthermore, Chen and Kwasik \cite{ChenKwasik-K3} rule out symplectic finite group actions acting trivially on homology on certain symplectic Calabi-Yau $4$-manifolds (see \cite{WuLiu} for a further development). Moreover, general constraints were obtained by Mundet i Riera \cite{Mundet-Ham}, showing roughly speaking that finite groups acting effectively in a  Hamiltonian way must be approximately abelian. 

Furthermore, Hamiltonian actions of cyclic groups on rational ruled symplectic $4$-manifolds, that is symplectic $S^2$-bundles over $S^2,$ were recently shown to be induced by $S^1$-actions \cite{ChiangKessler-ruled} (see \cite{Chen-Duke,Chen-Gokova, ChiangKessler-noext} for related and for contrasting results). Compatibly, the strongest restriction to date on manifolds admitting non-trivial Hamiltonian $S^1$-actions was obtained by McDuff \cite{McDuff-uniruled} who showed that all such manifolds must be {\em uniruled}, in the sense that at least one genus zero $k$-point Gromov-Witten invariant, for $k\geq 3,$ involving the point class must not vanish. Of course, rational ruled symplectic $4$-manifolds satisfy this condition, in fact with $k=3$: they are {\em strongly uniruled.} Either condition implies that these manifolds are {\em geometrically uniruled:} for each $\om$-compatible almost complex structure $J$ and each point $p\in M,$ there is a $J$-holomorphic sphere\footnote{This is a smooth map $u:\C P^1 \to M$ satisfying $Du\circ j = J \circ Du$ for the standard complex structure $j$ on $\C P^1.$ Such spheres and their significance in symplectic topology were discovered by Gromov \cite{GromovPseudohol}. We refer to McDuff-Salamon \cite{McDuffSalamon-BIG} for a detailed modern description of this notion.} passing through $p.$ Finally, in \cite{S-PRQS} a new notion of uniruledness, $\F_p$-Steenrod uniruledness, was introduced for $p=2,$ and was generalized to odd primes $p>2$ in \cite{SSW}: the quantum Steenrod $p$-th power of the cohomology class Poincar\'{e} dual to the point is defined and deformed in the sense of not coinciding with the classical Steenrod $p$-th power. This notion similarly implies geometric uniruledness. It is currently not known whether or not it implies uniruledness in the sense of McDuff, but it is expected that it does (see \cite{Seidel-formal,SeidelWilkins} for first steps in this direction).

In this paper we prove the first higher-dimensional Hamiltonian no-torsion results since that of Polterovich, that hold beyond the symplectically aspherical case. Firstly, we prove that, in addition to symplectically aspherical manifolds, symplectically Calabi-Yau manifolds and negative monotone symplectic manifolds do not admit Hamiltonian torsion. An elementary argument then shows, in summary, that if a closed symplectic manifold $M$ admits Hamiltonian torsion, then it admits a spherical homology class $A$ such that $\left<c_1(TM),A\right> > 0,$ $\left< [\om], A\right> > 0$ (see Corollary \ref{cor: spherical class}). Our results have a similar flavor to the result of McDuff for $S^1$-actions: indeed, negative monotone and Calabi-Yau manifolds are not geometrically uniruled, and neither are the symplectically aspherical ones. 

Going far beyond topological restrictions, we further study restrictions on Hamiltonian torsion in the positive monotone case. Using recently discovered techniques, we show that in this case the existence of non-trivial Hamiltonian torsion implies $\bF_p$ Steenrod-uniruledness for certain primes $p,$ and hence geometric uniruledness. This again fits well with the result of McDuff and in fact provides a partial solution to Problem 24 from the monograph \cite{McDuffSalamonIntro3} of McDuff-Salamon. Studying the properties of the quantum Steenrod operations and their relation to Gromov-Witten invariants further (see \cite{Wilkins,Wilkins-PSS, SeidelWilkins} for first inroads in this direction) might show that our solution is in fact quite complete. Furthermore, we are tempted to conjecture the following analogue of the result of McDuff.

\begin{conj}\label{conj: zp uniruled}
Each closed symplectic manifold with non-trivial Hamiltonian torsion must be uniruled.
\end{conj}

For the special case of symplectically Calabi-Yau $4$-manifolds, our results compare to those of Chen-Kwasik \cite{ChenKwasik-K3} as follows: the paper \cite{ChenKwasik-K3} rules out more general symplectic finite group actions than Hamiltonian ones, but in a more restrictive context, as they require these $4$-manifolds to have non-zero signature, and $b_2^+$ at least $2,$ while we make no such assumptions. Furthermore, the results of Wu-Liu \cite{WuLiu}, which hold only in dimension $4,$ cover a different class of manifolds than our results.


Before addressing further results on the metric properties of Hamiltonian torsion, when it exists in the positive monotone case, we comment on our methods of proof. Curiously enough, our arguments involve a recently discovered analogue of Smith theory in filtered Hamiltonian Floer homology \cite{Seidel-pants,SZhao-pants,S-HZ}, and related notions of quantum Steenrod powers \cite{Wilkins, Wilkins-PSS, SSW}. Previously these methods were applied to questions of existence of infinitely many periodic points \cite{S-HZ, SSW} and, more restrictively, of obstructions on manifolds to admit Hamiltonian pseudo-rotations \cite{S-PRQS, CGG2, S-PRQSR}. In fact, a general theme of this paper is that a Hamiltonian diffeomorphism of finite order behaves in many senses like a counterexample to the Conley conjecture. For example, the statement of  Corollary \ref{cor: spherical class} is analogous to that of \cite[Theorem 1.1]{GG-revisited} that provides the most general setting wherein the Conley conjecture is known to hold.

Additional results in this paper include the following. First, in the special case when $(M,\om)$ has minimal Chern number $N = n+1,$ we deduce from the work of Seidel and Wilkins \cite{SeidelWilkins}, as in \cite{S-PRQS}, that non-trivial Hamiltonian torsion implies that the quantum product $[pt] \ast [pt]$ does not vanish. This means that the manifold is strongly rationally connected: it implies strong uniruledness, and also shows that for each pair of distinct points $p_1, p_2$ in $M,$ and each $\om$-compatible almost complex structure $J,$ there exists a $J$-holomorphic sphere in $M$ passing through $p_1, p_2.$ Second, we prove that the spectral norm \cite{Schwarz:action-spectrum, Oh-specnorm,Viterbo-specGF} of a Hamiltonian torsion element $\phi$ of order $k$ on a closed rational symplectic manifold, that is $\left< [\om], \pi_2(M) \right> = \rho \cdot \Z,$ $\rho > 0,$ satisfies $\gamma(\phi) \geq \rho/k,$ and as an immediate consequence, the same estimate applies for the Hofer norm \cite{HoferMetric,Lalonde-McDuff-Energy}.

More importantly, in our final main result, we prove that in the monotone case, given $\phi \in \Ham(M,\om)\setminus\{\id\}$ of order $k,$ that is $\phi^k = \id,$ there exists $m \in \Z/k\Z$ such that \begin{equation}\label{eq: spectral Newman}\gamma(\phi^m) \geq \rho/3.\end{equation} 
This last result should be considered a Hamiltonian analogue of the celebrated result of Newman \cite{Newman} (see also \cite{Dress,SmithIII}), the $C^0$-distance having been replaced by the spectral distance. Moreover we prove the stronger statement that if $k$ is prime, then $\gamma(\phi^m) \geq \rho \lfloor k/2 \rfloor/ k$ for a certain $m \in \Z/k\Z,$ and provide a similar statement in the context of Hamiltonian pseudo-rotations.

The bound \eqref{eq: spectral Newman} can further be seen to imply Newman's result in a special case as follows. By \cite[Theorem C]{S-Zoll} (see also \cite{Kawamoto-C0}), when $M = \C P^n$ is the complex projective space with the standard symplectic form normalized so that $\C P^1$ has area $1,$ there is a constant $c_n$ depending only on the dimension, such that for all $\phi \in \Ham(M,\om),$ the usual $C^0$-distance of $\phi$ to the identity satisfies \[d_{C^0}(\phi,\id) \geq c_n \gamma(\phi).\] Hence if $\phi$ is of finite order, then by \eqref{eq: spectral Newman} there exists $m \in \Z$ such that \[d_{C^0}(\phi^m,\id) \geq  c_n/3.\]
It would be very interesting to see if the results of this paper can be generalized to the case of Hamiltonian homeomorphisms, as defined in \cite{BHS-spectrum}. This generalization does not seem to be straightforward because we use the properties of the linearization of the Hamiltonian diffeomorphism at its fixed points, as well as Smith theory in filtered Floer homology, which is not in general stable in the $C^0$-topology.

We close the introduction by noting that we expect that our results in the monotone case should generalize to the semi-positive case, once the relevant results of \cite{S-HZ},\cite{SSW} have been generalized to the requisite setting. Since these generalizations would not considerably differ, in a conceptual way, from the arguments presented in this paper, but would necessitate more lengthy technical proofs, we defer their investigation to further publications.



\subsection{Main results}

We start with the following theorem of Polterovich \cite{Pol-notors}, originally stated in the case where $\pi_2 = 0.$ For the reader's convenience we include its proof in Section \ref{subsec: theorem of Polterovich}. 

\begin{thm}[Polterovich]\label{thm: Polterovich}
Let $(M,\om)$ be a closed symplectically aspherical symplectic manifold. Then each homomorphism $G \to \Ham(M,\om),$ where $G$ is a finite group, is trivial.
\end{thm}

In this paper we prove a number of additional ``no-torsion" theorems of this kind, going beyond the symplectically aspherical case, and study the metric properties of Hamiltonian diffeomorphisms of finite order when such obstructions do not hold. Our conditions on the manifold that imply the absence of Hamiltonian torsion are of two kinds: the first is purely topological, and the second, perhaps more surprisingly, is in terms of pseudo-holomorphic curves.

\subsubsection{Topological conditions}

The first set of results of this paper is as follows.


\begin{thm}\label{thm: neg mon tors}
Let $(M,\om)$ be a closed negative monotone or closed symplectically Calabi-Yau symplectic manifold. Then each homomorphism $G \to \Ham(M,\om),$ where $G$ is a finite group, is trivial.
\end{thm}
%

A simple exercise in linear algebra shows that the class of manifolds, which we call {\em symplectically non-positive,} covered by Theorems \ref{thm: Polterovich} and \ref{thm: neg mon tors} can be described concisely as those closed symplectic manifolds $(M,\om)$ for which \[\left<[\om],A \right> \cdot \left<c_1(TM),A \right> \leq 0\] for all $A \in \pi_2(M).$ In other words, the following holds.

\begin{cor}\label{cor: spherical class}
	If a closed symplectic manifold $(M,\om)$ admits a non-trivial homomorphism $G \to \Ham(M,\om)$ from a finite group, then there exists $A \in \pi_2(M),$ such that $\langle [\om], A \rangle > 0,$ $\langle c_1(TM), A \rangle > 0.$ 
\end{cor}
For details of this implication see \cite[Proof of Theorem 4.1]{GG-revisited}.


Both these results are given by the following two steps that essentially generalize the notion of a perfect Hamiltonian diffeomorphism, namely one that has a finite number of contractible periodic points of all periods, to the case of compact path-connected isolated sets of fixed points. We call such an isolated set of fixed points of $\phi \in \Ham(M,\om)$ a {\em generalized fixed point} of $\phi.$ Recall that a fixed point $x$ of a Hamiltonian diffeomorphism $\phi = \phi^1_H$ is called contractible whenever the homotopy class $\al(x,\phi)$ of the path $\al(x,H)=\{\phi^{t}_H(x)\}$ for a Hamiltonian $H$ generating $\phi$, is trivial. This class does not depend on the choice of Hamiltonian by a classical argument in Floer theory. We call a generalized fixed point $\cl{F}$ of $\phi$ contractible, if all fixed points $x \in \cl{F}$ are contractible. 

We call $\phi \in \Ham(M,\om)$ {\em generalized perfect} if there exists a sequence $k_j \to \infty$ of iterations, such that for all $j \in \Z_{>0}$ the diffeomorphism $\phi^{k_j}$ has a finite number of contractible generalized fixed points, this set of contractible generalized fixed points does not depend on $j,$ and moreover the following condition regarding indices holds: for each capping $\ol{\mf{F}}$ of the generalized periodic orbit $\mf{F}$ corresponding to a generalized fixed point $\cl{F} \in \pi_0(\fix(\phi^{k_j})),$ and Hamiltonian $H$ generating $\phi,$ the mean-index $\Delta(H^{(k_j)}, \ol{x}),$ where $\ol{x}$ is a capped $1$-periodic orbit of $H^{(k_j)}$ corresponding to $x \in \cl{F}$ and capping $\ol{\mf{F}}$ is constant as a function of $x \in \cl{F}.$ We refer to Section \ref{subsubsec: mean-index} for the definition of the mean-index.

Finally we call a diffeomorphism $\phi$ with a finite number of (contractible) generalized fixed points {\em weakly non-degenerate} if for each (contractible) fixed point $x$ of $\phi,$ the spectrum of the differential $D(\phi)_x$ at $x$ contains points different from $1 \in \C.$ Using the existence of $\om$-compatible almost complex structures invariant under a Hamiltonian diffeomorphism of finite order, we prove the following structural result.

\begin{thm}\label{thm: torsion is perfect}
	Let  $(M,\om)$ be a closed symplectic manifold. Then a $p$-torsion Hamiltonian diffeomorphism $\phi \in \Ham(M,\om)$ for a prime $p,$ is a weakly non-degenerate generalized perfect Hamiltonian diffeomorphism. {In fact, it is Floer-Morse-Bott and its generalized fixed points are symplectic submanifolds.}
\end{thm}



Following the index arguments of Salamon-Zehnder \cite{SalamonZehnder}, and their generalization due to Ginzburg-G\"{u}rel \cite{GG-negmon}, we prove the following obstruction to the existence of weakly non-degenerate generalized perfect Hamiltonian diffeomorphisms. 

\begin{thm}\label{thm: generalized perfect CY negmon}
Let a closed symplectic manifold $(M,\om)$ be negative monotone or symplectically Calabi-Yau. Then $(M,\om)$ does not admit weakly non-degenerate generalized perfect Hamiltonian diffeomorphisms.
\end{thm}

Together with Theorem \ref{thm: torsion is perfect}, Theorem \ref{thm: generalized perfect CY negmon} immediately implies Theorem \ref{thm: neg mon tors}, after the elementary observation that in view of Cauchy's theorem for finite groups, to rule out all Hamiltonian finite group actions it is sufficient to rule out all Hamiltonian torsion of prime order. One can say that, almost paradoxically, we use the {\em large time asymptotic behavior} of our Hamiltonian system to study its {\em periodic} dynamics! This can be considered to be the main general idea of this paper.

We remark that, as easy examples show, generalized perfect Hamiltonian diffeomorphisms do indeed exist on the manifolds of Theorem \ref{thm: generalized perfect CY negmon}, if one drops the weak non-degeneracy assumption. For example, one can take $T^2 = S^1 \times S^1,$ $S^1 = \R/\Z$ to be the standard torus with $(x,y)$ denoting a general point, and $\om_{st} = dx \wedge dy$ the standard symplectic form, and pick $\phi \in \Ham(T^2,\om_{st}),$ $\phi = \phi^t_H,$ $t>0,$ for $H \in C^{\infty}(T^2,\R)$ given by $H(x,y) = \cos(2\pi y).$ It is easy to see that the set of contractible periodic points of $\phi$ consists precisely of the two isolated sets $\{y = 0 \},$ and $\{y = \frac{1}{2}\}.$


\subsubsection{Conditions in terms of pseudo-holomorphic curves}

Our second set of results deals with the class of monotone symplectic manifolds. It is evident that far more than topological conditions are necessary to rule out Hamiltonian torsion in this case, since each Hamiltonian $S^1$-manifold, such as $\C P^n$ for example, admits Hamiltonian torsion. We formulate our restriction on the existence of Hamiltonian torsion geometrically as follows. For an $\om$-compatible almost complex structure $J$ on a closed symplectic manifold $(M,\om)$ we say that the manifold is {\em geometrically uniruled} if for each point $p \in M,$ there exists a $J$-holomorphic sphere $u:\C P^1 \to M,$ such that $p \in \ima(u).$

\begin{thm}\label{thm: torsion to geom uniruled}
Let $(M,\om)$ be a closed monotone symplectic manifold, that is not geometrically uniruled for a certain $\om$-compatible almost complex structure $J.$ Then each homomorphism $G \to \Ham(M,\om),$ where $G$ is a finite group, is trivial. 
\end{thm}

This is a corollary of the following more precise result involving the quantum Steenrod power operations.

\begin{thm}\label{thm: St uniruled}
Let $(M,\om)$ be a closed monotone symplectic manifold that admits a Hamiltonian diffeomorphism of order $d > 1.$ Then the $p$-th quantum Steenrod power of the cohomology class $\mu \in H^{2n}(M; \bF_p)$ Poincar\'{e} dual to the point class is deformed for all primes $p$ coprime to $d.$ 
\end{thm}

Theorems \ref{thm: torsion to geom uniruled} and \ref{thm: St uniruled} provide an obstruction to the existence of Hamiltonian diffeomorphisms of finite order in terms of pseudo-holomorphic curves. The existence of an obstruction of this type was conjectured by McDuff and Salamon, and publicized as Problem 24 in their introductory monograph \cite{McDuffSalamonIntro3}. Therefore we provide a solution to a reasonable variant of this problem. Indeed, further investigations into the enumerative nature of quantum Steenrod operations might prove that our solution is in fact complete, in the framework of monotone symplectic manifolds. Such investigations were initiated in \cite{Wilkins,Wilkins-PSS,SeidelWilkins}.




%


The proof of Theorem \ref{thm: St uniruled} relies on two steps. These steps are aimed at showing that torsion Hamiltonian diffeomorphisms of closed monotone symplectic manifolds, that by Theorem \ref{thm: torsion is perfect} are generalized perfect and weakly non-degenerate, are moreover homologically minimal in the following sense. To formulate it precisely, we first discuss a useful technical notion.   

Let $\bK$ be a coefficient field. For a generalized fixed point $\cl{F}$ of a Hamiltonian diffeomorphism $\psi$ we define a generalized version $HF^{\loc}(\psi,\cF)$ of local Floer homology. Such notions date back to the original work of Floer \cite{Floer-MorseWitten,Floer3} and have been revisted a number of times: for example by Pozniak in \cite{Pozniak}. It is naturally $\Z/2\Z$-graded\footnote{We also define a $\Z$-graded version for a capped generalized $1$-periodic point $\ol{\fF}$ lifting $\cF.$}.

We call a Hamiltonian diffeomorphism a {\em generalized $\bK$ pseudo-rotation} with the sequence $k_j$ if it is generalized perfect with the sequence $k_j,$ and furthermore $HF^{\loc}(\psi,\cl{F}) \neq 0$ for all $\cl{F} \in \pi_0 (\fix(\psi)),$ and the homological count \[N(\psi,\bK) := \sum_{\cF \in \pi_0 (\fix(\psi))} \dim_{\bK} HF^{\loc}(\psi,\cF)\] of generalized fixed points of $\psi = \phi^{k_j}$ satisfies \[ N(\psi,\bK) = \dim_{\bK} H_*(M;\bK)\] for all $j \in \Z_{>0}.$ We recall that usually an $\F_p$ pseudo-rotation is defined analogously, with the sequence $k_j = p^{j-1}$ with the additional hypothesis that each $\cl{F} \in \pi_0(\fix(\psi))$ for $\psi = \phi^{k_j}$ consists of a single point. We note that $k_j = p^{j-1}$ shall often be a convenient sequence for us to work with.

In view of the discussion in \cite{S-HZ, S-Zoll} this homological minimality for a Hamiltonian diffeomorphism $\psi$ with a finite number of generalized fixed points is equivalent to the absence of finite bars in the barcode $\cB(\psi)$ of $\psi,$ a notion of recent interest in symplectic topology (see \cite{PolShe,PolSheSto,KS-bounds,S-HZ, S-Zoll} for example). It also implies the equality \[\Spec^{ess}(F;\bK) = \Spec^{vis}(F;\bK)\] between two homologically defined subsets of the spectrum associated to a Hamiltonian $F \in \cl{H}$ generating $\psi.$ Recall that the spectrum $\Spec(F)$ of $F$ is the set of critical values of the action functional of $F.$ For a coefficient field $\bK,$ there is a nested sequence of subsets \[\Spec^{ess}(F;\bK) \subset \Spec^{vis}(F;\bK) \subset \Spec(F).\] Here the {\em essential spectrum} $\Spec^{ess}(F;\bK)$ is the set of values of all spectral invariants associated to $F,$ in other words the set of starting points of infinite bars in the barcode of $F,$ and the {\em visible spectrum} $\Spec^{vis}(F;\bK)$ is the set of action values of capped (generalized) periodic orbits of $F$ that have non-zero local Floer homology, in other words the set of endpoints of all bars in the barcode. It is not hard to modify the definitions of the two homological spectra to include multiplicities, in which case their equality would be equivalent to homological minimality.

The first step in the proof of Theorem \ref{thm: St uniruled}, which is non-trivial and uses the main methods of \cite{S-HZ}, is the following reduction.


\begin{thm}\label{thm: torsion is PR}
Let $(M,\om)$ be a closed monotone symplectic manifold. Suppose that $\phi \in \Ham(M,\om)$ is a Hamiltonian diffeomorphism of prime order $q \geq 2.$ Then:

\begin{enumerate}[label = (\roman*)]
\item \label{G: case 1} For each prime $p$ different from $q,$ the $q$-torsion diffeomorphism $\phi$ is a weakly non-degenerate generalized pseudo-rotation over $\F_p,$ with the sequence $k_j$ given by the monotone-increasing ordering of any infinite subset of the set \[\{k \in \Z\;|\; k \neq 0 \;(\mathrm{mod}\; q) \}.\]
\item \label{G: case 2} Moreover \[\Spec^{ess}(H; \bK) = \Spec^{vis}(H;\bK)\] for each Hamiltonian $H$ generating $\phi,$ and each coefficient field $\bK,$ of characteristic $p \neq q$ and \[\Spec^{ess}(H^{(k)}; \Q) = k\cdot \Spec^{ess}(H; \Q) + \rho \cdot \Z\] for all $k$ coprime to $q.$ 
\item  \label{G: case 3} Finally, part \ref{G: case 1} holds also for $p=q,$ and in part \ref{G: case 2} \[\Spec^{ess}(H; \bK) = \Spec^{vis}(H;\bK),\] \[\Spec^{ess}(H^{(k)}; \bK) = k\cdot \Spec^{ess}(H; \bK) + \rho \cdot \Z\] hold with arbitary coefficient field $\bK,$ and moreover \[\Spec^{vis}(H;\bK) = \Spec(H).\] 
\end{enumerate}

 

\end{thm}  

We briefly explain the approach used to prove part \ref{G: case 1} of Theorem \ref{thm: torsion is PR}. Following the main theme of the proof of Theorem \ref{thm: torsion is perfect}, we use information about large iterations of $H$ to study the periodic Hamiltonian diffeomorphism $\phi = \phi^1_H$ that it generates. More precisely, let $\psi = \phi^k,$ for $k$ coprime to $q.$ Combining the theory of barcodes of Hamiltonian diffeomorphisms (see Proposition \ref{prop: barcodes}), and Smith-type inqualities in filtered Floer homology (see Theorem \ref{thm: Smith}), we observe that for the {\em bar-lengths} \[\beta_1(\psi,\F_p) \leq \ldots \leq \beta_{K(\psi,\F_p)}(\psi,\F_p)\] of $\psi,$ we have the following inequality. Set \[\beta_{\mrm{tot}}(\psi,\F_p) = \beta_1(\psi,\F_p) + \ldots + \beta_{K(\psi,\F_p)}(\psi,\F_p)\] to be the {\em total bar-length} of $\psi.$ Then \[ \beta_{\mrm{tot}}(\psi^{p^m},\F_p) \geq p^m \cdot \beta_{\mrm{tot}}(\psi,\F_p).\] However, $\beta_{\mrm{tot}}(\psi^{p^m},\F_p)$ is bounded, since it can take at most $q-1$ values. This implies that \[\beta_{\mrm{tot}}(\psi,\F_p) = 0\] which in turn, implies part \ref{G: case 1}, by the theory of barcodes (see Proposition \ref{prop: barcodes}).

\begin{rmk}
We separate part \ref{G: case 3} of Theorem \ref{thm: torsion is PR} because it requires a different proof, relying on Proposition \ref{prop: Pozniak} below. The first statement of part \ref{G: case 3} is obtained via Proposition \ref{prop: Pozniak} by classical Smith theory combined with classical homological Arnol'd conjecture (outlined in \cite[Remark 7.1]{ChiangKessler-ruled} with details for $p=2$). One could also obtain this statement by a suitable generalization of Theorem \ref{thm: Smith} on Smith theory in filtered Floer homology, which is however out of the scope of this paper. 
\end{rmk}

The following statement is a key component of part \ref{G: case 3} of the proof of Theorem \ref{thm: torsion is PR}. It relies on the generalization of the Morse-Bott theory of Pozniak \cite[Theorem 3.4.11]{Pozniak} to the situation with signs and orientations, as in for example \cite[Chapter 9]{Schmaschke}, \cite[Chapter 8]{FO3:book-vol12}, or \cite{WehrheimWoodward-orientations}. However it is not entirely straightforward, because as classical examples show, it is false in the general Floer-Morse-Bott situation, while in our case it holds because of the existence of special $\om$-compatible almost complex structures adapted to the situation. 

\begin{prop}\label{prop: Pozniak}
Let $(M,\om)$ be a closed symplectic manifold, and $\phi \in \Ham(M,\om)$ a Hamiltonian diffeomorphism of finite order $d \geq 2.$ Let $\cl{F}$ be a path-connected component of the fixed-point set of $\phi.$ Finally, let $R$ be a commutative unital ring. Then the local Floer homology of $\phi$ at $\cl{F}$ with coefficients in $R$ satisfies: \[ HF^{\loc}(\phi, \cl{F}) \cong H(\cl{F}; R).\] 
\end{prop}

We also note that the proof of Theorem \ref{thm: torsion is PR} has the following by-product, which is a new analogue, for Hamiltonian torsion, of the classical consequence of Floer theory, whereby the map $\pi_1(\Ham(M,\om)) \to \pi_1(M)$ is trivial.  

\begin{prop}
Let $(M,\om)$ be a closed monotone symplectic manifold, and $\phi$ in $\Ham(M,\om)$ be a Hamiltonian diffeomorphism of prime order. Then {\em all} the fixed points of $\phi$ are contractible.
\end{prop}

The second step in the argument proving Theorem \ref{thm: St uniruled} is the following statement. It essentially follows the arguments of \cite{S-PRQSR} and \cite{SSW}.

\begin{thm}\label{thm: Morse-Bott PR uniruled}
	Let $(M,\om)$ be a closed monotone symplectic manifold that admits a weakly non-degenerate generalized $\F_p$ pseudo-rotation for a prime $p \geq 2.$ Then the $p$-th quantum Steenrod power of the cohomology class $\mu \in H^{2n}(M; \bF_p)$ Poincar\'{e} dual to the point class is deformed.
\end{thm}

These two results immediately imply Theorem \ref{thm: St uniruled} and therefore, by a Gromov compactness argument, Theorem \ref{thm: torsion to geom uniruled}.


\subsubsection{Applications to actions of Lie groups and lattices}
To conclude the discussion of our first two sets of results, we discuss their implications to the question of existence of Hamiltonian actions of possibly disconnected Lie groups and lattices in Lie groups, on closed symplectic manifolds. 

A well-known result of Delzant \cite{Delzant} (see \cite{PolterovichRosen} for an alternative argument) implies that a connected semi-simple group can only act on a closed symplectic manifold if it is compact. A compact zero-dimensional Lie group is finite, whence Theorems \ref{thm: neg mon tors} and \ref{thm: torsion to geom uniruled} provide topological and geometrical obstructions to their action. The identity component $K_0$ of a compact Lie group $K$ of positive dimension is a compact connected Lie group of positive dimension, and as such admits a maximal torus $T \cong (S^1)^k$ of positive dimension whose conjugates cover the whole group. Therefore, the absence of Hamiltonian torsion, as in Theorems \ref{thm: Polterovich}, \ref{thm: neg mon tors}, \ref{thm: torsion to geom uniruled}, and \ref{thm: St uniruled}  implies that a non-trivial $K$-action yields a non-trivial $K_0$-action, since otherwise it would factor through $K/K_0$ which is finite. This in turn yields a non-trivial $T$-action and a fortiori a non-trivial $S^1$-action. A celebrated result of McDuff \cite{McDuff-uniruled} then shows that non-trivial $S^1$-actions imply uniruledness in the sense of $k$-point genus $0$ Gromov-Witten invariants and hence geometric uniruledness. 

\begin{cor}[McDuff]\label{thm: McDuff}
	Let $(M,\om)$ be a closed symplectic manifold that is not geometrically uniruled. Then each homomorphism $K_0 \to \Ham(M,\om)$ for a compact connected Lie group $K_0$ must be trivial.
\end{cor}

Moreover, by a simple continuity argument, a non-trivial continuous $S^1$-action implies a non-trivial $\Z/p\Z$-action for each prime $p.$ Therefore Theorems \ref{thm: neg mon tors} and \ref{thm: torsion to geom uniruled} imply the above corollary of the result of McDuff for symplectically aspherical, symplectically Calabi-Yau, negative monotone, or monotone symplectic manifolds. However, Theorem \ref{thm: St uniruled} also implies that in this case if the manifold is monotone, it must be $\F_p$ Steenrod-uniruled for all primes $p.$ It is seen from examples due to Seidel and Wilkins \cite{SeidelWilkins} that there exist closed monotone symplectic manifolds that are uniruled in the sense of Gromov-Witten invariants, and yet not $\F_p$ Steenrod-uniruled for certain primes $p.$ 

We note that McDuff's theorem was proven by showing that certain loops of Hamiltonian diffeomorphisms in a blow-up of the manifold are non-trivial, and detectable by Seidel's representation \cite{seidelInvertibles}. It would be interesting to investigate the existence of non-trivial Hamiltonian loops associated to Hamiltonian diffeomorphisms of finite order. We note that for a Hamiltonian $H$ generating $\phi \in \Ham(M,\om)$ of order $d,$ the Hamiltonian $H^{(d)}$ generates a loop homotopic to $\{(\phi^t_H)^d\}.$ The non-contractibility of this loop is not obvious since for a rotation $\phi_{2\pi/3}$ of $S^2$ by angle $2\pi/3$ about the $z$-axis, the loop $\{ \phi_{t\cdot2\pi/3}^3\}$ is not contractible in $\Ham(S^2,\om_{st}),$ while the loop $\{ \phi_{- t\cdot 4\pi/3}^3\}$ is contractible therein, yet $\phi_{-4\pi/3} = \phi_{2\pi/3}.$


Finally we can argue, following the work of Polterovich \cite{Pol-notors} on the Hamiltonian Zimmer conjecture, that $SL(k,\Z)$ for $k\geq 2$ has no non-trivial Hamiltonian actions on symplectically aspherical, symplectically Calabi-Yau, negative monotone, or monotone and not geometrically uniruled closed symplectic manifolds. Indeed, it is well-known that $SL(k,\Z)$ for $k \geq 2$ is generated by elements of finite order. We remark, however, that the case of finite index subgroups of $SL(k,\Z),$ $k\geq 3,$ is much more difficult and seems to be currently out of reach of our methods.



\subsubsection{Metric properties}
Our third and final set of results studies the metric properties of Hamiltonian torsion diffeomorphisms, in cases that are not ruled out by our previous arguments, for example on $\C P^n.$

Recall that the \textit{spectral pseudo-norm} of a Hamiltonian $H \in \sm{S^1 \times M, \R}$ on a closed symplectic manifold $(M,\om)$ is defined in terms of Hamiltonian spectral invariants as \[ \gamma(H) = c([M],H) + c([M],\ol{H}),\] and the {\it spectral norm} of $\phi \in \Ham(M,\om)$ is set as \[\gamma(\phi) = \inf_{\phi_H^1 = \phi} \gamma(H).\] We refer to Section \ref{sec: prelims} for a more in-depth discussion of this interesting notion, remarking for now that this is a conjugation-invariant and non-degenerate norm on $\Ham(M,\om),$ yielding a bi-invariant metric \[d_{\gamma}(f,g) = \gamma(gf^{-1}).\] This was shown in large generality in \cite{Oh-specnorm,Viterbo-specGF,Schwarz:action-spectrum}.

Furthermore, whenever defined, $\gamma(\phi)$ provides a lower bound on the celebrated Hofer distance \cite{HoferMetric, Lalonde-McDuff-Energy} $d_{\rm{Hofer}}(\phi, \id),$ defined as \[d_{\rm{Hofer}}(\phi, \id) = \inf_{\phi_H^1 = \phi} \intoi \max_M H(t,-) - \min_M H(t,-) \; dt.\] 

Finally in \cite{BHS-spectrum, Kawamoto-C0, S-Zoll} it was shown in various degrees of generality that $\gamma(\phi)$ is bounded by the $C^0$-distance $d_{C^0}(\phi,\id)$ of $\phi$ to the identity, at least in a small $d_{C^0}$-neighborhood of the identity. 

These and numerous other recent results show that the spectral norm $\gamma$ is an important measure of a Hamiltonian diffeomorphism. Here, we provide lower bounds on $\gamma(\phi),$ under the assumption that $\phi$ is of finite order. Our first result is relatively general and quite straightforward, and follows essentially from the homogeneity of the action under iteration, however it underlines the fact that the finite order condition implies certain metric rigidity.  


\begin{thm}\label{thm: lower bound}
	Let $(M,\om)$ be a closed rational symplectic manifold, with rationality constant $\rho > 0,$ that is $\langle [\om], \pi_2(M)\rangle = \rho \cdot \Z.$  Let $\phi \in \Ham(M,\om)$ be a non-trivial Hamiltonian diffeomorphism of order $d,$ that is $\phi^d = \id.$ Then $\gamma(\phi) \geq \rho/d.$
\end{thm}

As a further consequence of Theorem \ref{thm: torsion is PR}, which requires considerably more complex methods, we obtain the following analogue of Newman's theorem for the spectral norm of Hamiltonian torsion elements.  

\begin{thm}\label{thm: spectral Newman}
Let $(M,\om)$ be a closed monotone symplectic manifold. Consider  a Hamiltonian diffeomorphism $\phi \in \Ham(M,\om)$ of order $d > 1.$ Then if the rationality constant of $\om$ is $\rho > 0,$ there exists $m \in \Z/d\Z$ such that \[ \gamma(\phi^m) \geq \rho/3.\] Here the coefficients are in an arbitrary field $\bK.$ In fact, if $d=p$ is prime, we prove the stronger statement, that there exists $m \in \Z/p\Z,$ such that \[ \gamma(\phi^m)  \geq \rho \cdot \lfloor p/2 \rfloor /p.\]
\end{thm}  

The key notion in the proof of this result is a new invariant of a Hamiltonian diffeomorphism $\phi \in \Ham(M,\om),$ which we call the {\em spectral length} $l(\phi,\bK)$ of $\phi$ with coefficients in a field $\bK.$ It is defined as the minimal diameter of $\Spec^{ess}(H;\bK) \cap I,$ for an interval $I =(a-\rho,a] \subset \R$ of length $\rho$ (with fixed $H$ as $I$ varies). In particular, we show that this minimum does not depend on the choice of Hamiltonian $H$ with $\phi^1_H = \phi.$ We show the key property that $l(\phi,\bK) \leq \gamma(\phi,\bK)$ and that in our case the spectral length behaves in a controlled way with respect to iterations. By a combinatorial analysis of our situation we consequently deduce Theorem \ref{thm: spectral Newman}.
We expect $l(\phi,\bK)$ to have additional applications in quantitative symplectic topology that we plan to investigate.

Theorem \ref{thm: spectral Newman} is generally speaking sharp, as can be seen from the rotation $\phi$ of $S^2$ by $2\pi/3$ about the $z$-axis. In this case $\phi^3 = \id$ and $\gamma(\phi) = \gamma(\phi^2) = \gamma(\phi^{-1}) = \rho/3,$ where $\rho$ is the area of the sphere. Observe moreover that the lower bound in Theorem \ref{thm: spectral Newman} does not depend on the order of $\phi.$ In particular if $d=2$ then Theorem \ref{thm: lower bound} gives the stronger lower bound $\gamma(\phi) \geq \rho/2,$ which is again sharp for the $\pi$-rotation of $S^2$ about the $z$-axis. We recall that Newman's theorem is the same assertion, but for the $C^0$-distance to the identity, in the setting of homeomorphisms of smooth manifolds. In contrast to our result, the constant in Newman's theorem is not explicit.

Finally, we remark that analogous statements hold for generalized $\F_p$ pseudo-rotations $\phi$ with sufficiently large admissible sequences. For example, for the sequence $k_j = p^{j-1},$ we get the lower bound $\gamma(\phi^{k_j}) \geq \rho/(p+1)$ for some $j \in \Z_{>0},$ which is saturated by the rotation of $S^2$ by $2\pi/(p+1)$ about the $z$-axis. For the sequence $k_j = j,$ we obtain the following lower bound, which is saturated by any $2\pi \theta$-rotation on $S^2$ about the $z$-axis, where $\theta \notin \Q.$ 

\begin{thm}\label{thm: PR}
	Let $\phi \in \Ham(M,\om)$ be a generalized $\bK$ pseudo-rotation with sequence $k_j = j$ on a closed monotone symplectic manifold $(M,\om)$ with rationality constant $\rho.$ Then \[\sup_{j \in \Z_{>0}} \gamma(\phi^{k_j}) \geq \rho/2,\] the coefficients being taken in $\bK.$
\end{thm}

This result is new in this generality even for strongly non-degenerate pseudo-rotations. In the special case where $(M,\om)$ is a complex projective space, this result can also be obtained in a different way by following the methods of \cite{GG-pseudorotations}.




%


\section{Preliminary material}\label{sec: prelims}

\subsection{Basic setup}

In this section, we recall established aspects of the theory of Hamiltonian diffeomorphisms on symplectic manifolds. Throughout the article, $(M,\om)$ denotes a $2n$-dimensional closed symplectic manifold.

\begin{df}[Monotone, negative monotone and symplectically Calabi-Yau]
	Suppose that the cohomology class of the symplectic form $\om$ is proportional to the first Chern class
	\begin{equation*}
	[\om] = \kappa\cdot c_1(TM),
	\end{equation*}
	for $\kappa\neq 0$, on the image $H^{S}_{2}(M;\Z)$ of the Hurewicz map $\pi_2(M) \to H_2(M;\Z).$ If $\kappa <0$ we call $(M,\om)$ \textit{negative monotone}, and if $\kappa >0$  we simply say \textit{monotone}. If the first Chern class $c_{1}(TM)$ vanishes on the image of the Hurewicz map, we say $(M,\om)$ is \textit{symplectically Calabi-Yau}.
\end{df}

The symplectic manifold $(M,\om)$ is called \textit{rational} whenever $\cP_{\om}=\brat{[\om],H^{S}_{2}(M;\Z)}$ is a discrete subgroup of $\R$. If $\cP_{\om}\neq\{0\}$, then $\cP_{\om}=\rho\cdot \Z$ for $\rho>0$, which we call the \textit{rationality constant} of $(M,\om)$. If $\cP_{\om}={0}$ we call $(M,\om)$ \textit{symplectically aspherical}. 

Finally we recall that the minimal Chern number of $(M,\om)$ is the index \[N = N_M = [\Z:\im(c_1(TM): \pi_2(M) \to \Z)].\]


\subsubsection{Hamiltonian isotopies and diffeomorphisms}
We consider normalized 1-periodic Hamiltonian functions $H\in\cl{H}\subset\sm{\rS^1 \times M, \R}$, where $\cl{H}$ is the space of Hamiltonians normalized so that $H(t,-)$ has zero $\om^n$ mean for all $t \in [0,1].$ For each $H\in\cl{H}$ we have the corresponding time-dependent vector field $X_H^t$ defined by the relation $\om(X_H^t,\cdot\,) = - dH_t$. In particular, to each Hamiltonian function we can associate a Hamiltonian isotopy $\{\phi_H^t\}$ induced by $X_H^t$ and its time-one map $\phi_H=\phi_H^1$. We omit the $H$ from this notation whenever it is clear from context. Such maps $\phi_H$ are called Hamiltonian diffeomorphisms and they form a group denoted by Ham$(M,\om)$. 

For a Hamiltonian diffeomorphism $\phi\in\Ham(M,\om)$, we denote the set of its \textit{contractible} fixed points by $\fix(\phi)$, and by $x^{(k)}$, for $x\in\fix(\phi)$, its image under the inclusion $\fix(\phi)\subset\fix(\phi^k)$. Contractible means the homotopy class $\al(x,\phi)$ of the path $\al(x,H)=\{\phi^{t}_H(x)\}$ for a Hamiltonian $H\in\cH$ generating $\phi$, is trivial. This class does not depend on the choice of Hamiltonian by a classical argument in Floer theory.

We denote by $H^{(k)}$ the $1$-periodic $k$-{th} iteration of a Hamiltonian function $H$, where $H^{(k)}(t,x)=kH(kt,x)$ and $\phi_{H^{(k)}}=\phi_H^k$. There is a bijective correspondence between $\fix(\phi_{H})$ and contractible $1$-periodic orbits of the isotopy $\{\phi^{t}_{H}\}$, thus for $x\in\fix(\phi_{H})$, we denote by $x(t)$ the $1$-periodic orbit given by $x(t)=\phi^{t}_{H}(x)$ and, similarly, by $x^{(k)}(t)$ the $1$-periodic orbit given by $x^{(k)}(t)=\phi^{t}_{H^{(k)}}(x^{(k)})$.


\subsubsection{The Hamiltonian action functional} Let $x:\rS^1\ra M$ be a contractible loop. It is then possible to extend this map to a \textit{capping} of $x$, namely, a map $\ol{x}:\mathrm{D}^2\ra M$ such that $\ol{x}|_{\rS^1}=x$. Let $\cL_{pt}M$ denote the space of contractible loops in $M$ and consider the equivalence relation on capped orbits given by
\begin{equation*}
(x,\ol{x}) \sim(y,\ol{y}) \iff x = y \quad\text{and}\quad \ol{x}\#(-\ol{y}) \in\ker[\om],
\end{equation*}
where $\ol{x}\#(-\ol{y})$ stands for gluing disks along their boundaries with the orientation of $\ol{y}$ reversed. The quotient space $\til\cL_{pt}M$ of capped orbits by the above equivalence relation  is a covering of $\cL_{pt}M$ with the group of deck transformations isomorphic to $\Gamma = \pi_2(M)/\ker[\om]$. We write $(x,\ol{x})$ or simply $\ol{x}$ for the equivalence class of the capped orbit. With this notation, to each $A\in\Gamma$ we associate the deck transformation sending a capped orbit $\ol{x}$ to $\ol{x}\#A$. We define the Hamiltonian action functional $\cA_H:\til\cL_{pt}M\ra\R$ of a $1$-periodic Hamiltonian $H$ by 
\begin{equation*}
\cA_H(\ol{x}) = \int_0^1 H(t,x(t))dt -\int_{\ol{x}}\om.
\end{equation*}
Observe that the critical points of the Hamiltonian action functional are exactly $(x,\ol{x})$ for $x$ a contractible 1-periodic orbit $\cO(H)$ satisfying the equation $x'(t) = X_H^t(x(t))$. We denote by $\til{\cO}(H)$ the set of critical points of $\cA_H$. The \textit{action spectrum} of $H$ is defined as $\spec(H)=\cA_{H}(\til{\cO}(H))$. We remark following \cite{Schwarz:action-spectrum} that in the monotone and negative monotone cases the action spectrum is a closed nowhere dense subset of $\R$. In addition, if $A\in\Gamma$ then 
\begin{equation*}
\cA_H(\ol{x}\#A)= \cA_H(\ol{x}) - \int_A\om,
\end{equation*}
and for $\ol{x}^{(k)}$, the $k-$th iteration of $\ol{x}$ with the naturally inherited capping, we have
\begin{equation*}
\cA_{H^{(k)}}(\ol{x}^{(k)}) = k\cA_H(\ol{x}).
\end{equation*}

\begin{df}[Non-degenerate and weakly non-degenerate orbits]
	A $1$-periodic orbit $x$ of $H$ is called \textit{non-degenerate} if $1$ is not an eigenvalue of the linearized time-one map $D(\phi^{1}_{H})_{x(0)}$. We call $x$ \textit{weakly non-degenerate} if there exists at least one eigenvalue of $D(\phi^{1}_{H})_{x(0)}$ different from $1$. A Hamiltonian $H$ is said \textit{non-degenerate} if all its $1$-periodic orbits are non-degenerate. 
\end{df}

The non-degeneracy of an orbit $x$ of $H$ is equivalent to \[graph(\phi_H)=\{(x,\phi_H( x))|\,x\in M\}\] intersecting the diagonal $\Delta_M\subset M\times M$ transversely at $(x(0),x(0))$. Following \cite{SalamonZehnder} for any Hamiltonian $H$ and $\eps>0$ there exists a non-degenerate Hamiltonian $H'$ satisfying $\|H-H'\|_{C^2} < \eps$. This fact is key in the definition of filtered Floer Homology of degenerate Hamiltonians and for local Floer homology.

\subsubsection{Mean-index and the Conley-Zehnder index}\label{subsubsec: mean-index}


Following \cite{SalamonZehnder,GG-negmon}, the \textit{mean-index} $\Delta(H,\ol{x})$ of a capped orbit $\ol{x}$ of a possibly degenerate Hamiltonian $H$ is a real number measuring the sum of the angles swept by the eigenvalues of $\{D(\phi^{t}_H)_{x(t)}\}$ lying on the unit circle. Here a trivialization induced by the capping is used in order to view $\{D(\phi^{t}_H)_{x(t)}\}$ as a path in $Sp(2n,\R)$. One can show that for the time-one map $\phi=\phi_H$ generated by the Hamiltonian $H$ the mean-index depends only on the class $\til{\phi}$ of $\{\phi_H^t\}_{t\in[0,1]}$ in the universal cover $\til{\Ham}(M,\om)$ making the notation $\Delta(\til{\phi},\ol{x})$ suitable. In addition, the mean-index depends continuously on both $\phi$ and the capped orbit and it behaves well with iterations, 
\begin{equation*}
\Delta(\til\phi^k,\ol{x}^{(k)}) = k\cdot\Delta(\til\phi,\ol{x}).
\end{equation*}
Meanwhile, the \textit{Conley-Zehnder index} CZ$(H,\ol{x})$ of a non-degenerate Hamiltonian is integer valued and roughly measures the winding number of the above mentioned eigenvalues. Once again, the index only depends on $\til{\phi}$ so we can also write CZ$(H,\ol{x})=$ CZ$(\til\phi,\ol{x})$. We shall use the same normalization as in \cite{GG-negmon}, namely, CZ$(H,\ol{x})= n$ if $x$ is a non-degenerate maximum of an autonomous Hamiltonian $H$ with small Hessian and $\ol{x}$ is the constant capping. We shall omit the $H$ or $\til\phi$ in the notation when it is clear from the context. We remark that for an element $A\in\Gamma$ 
\begin{equation*}
\Delta(\ol{x}\#A) = \Delta(\ol{x}) -2\brat{c_1(TM),A}\quad\text{and}\quad \text{CZ}(\ol{x}\#A) = \text{CZ}(\ol{x})-2\brat{c_1(TM),A}.
\end{equation*}
Also, in the case that $\ol{x}$ is non-degenerate we have
\begin{equation}\label{mi_inqty}
|\Delta(\ol{x})-\text{CZ}(\ol{x})| < n. 
\end{equation}
Following \cite{RobbinSalamonIndex,PolterovichRosen,EntovPolterovich-rigid} we observe that a version of the Conley-Zehnder index can be defined even in the case where the capped orbit is degenerate. It is called the Robbin-Salamon index and it coincides with the usual Conley-Zehnder index in the non-degenerate case. Furthermore we note that the mean-index can be equivalently defined by 
\begin{equation}\label{eq: limit CZ}
     \Delta(\til{\phi}_H,\ol{x}) = \lim_{k \to \infty} \frac{1}{k} \text{CZ}(\til{\phi}^{k}_{H}, \ol{x}^{(k)}),
\end{equation}
where we are slightly abusing notation in the sense that CZ here means the Robbin-Salamon index so as to include the degenerate case. The limit in \eqref{eq: limit CZ} exists, as the Robbin-Salamon index is a quasi-morphism $CZ: \til{Sp}(2n,\R) \to \R$ (see \cite[Section 3.3.4]{EntovPolterovich-rigid}, \cite{GGP-CZ}). In particular, as can also be seen directly from its definition in \cite{SalamonZehnder}, the mean-index is induced by a homogeneous quasi-morphism $\Delta: \til{Sp}(2n,\R) \to \R.$ Morevoer, this map is continuous, and satisfies the additivity property \[\Delta(\Phi \Psi) = \Delta(\Phi) + \Delta(\Psi)\] for all $\Phi \in \pi_1(Sp(2n,\R)) \subset \til{Sp}(2n,\R)$ and all $\Psi \in \til{Sp}(2n,\R).$


\subsection{Floer theory} 
Floer thoery was first introduced by A.Floer \cite{Floer1,Floer2,Floer3} as a generalization of the Morse-Novikov homology for the Hamiltonian action functional defined above. We refer to \cite{OhBookI} and references therein for details on the constructions described in this subsection, and to \cite{AbouzaidBook,SeidelMCG,Zap:Orient}, as well as to references therein, for a discussion of canonical orientations. 

\subsubsection{Filtered and total Floer Homology}
In this subsection we review the construction of \textit{filtered Floer homology} in order to recall some basic properties and set notation. 

Let $H$ be a non-degenerate $1$-periodic Hamiltonian on a rational symplectic manifold $(M,\om)$ and $\bK$ a fixed base field. For $a\in\R\setminus$Spec$(H)$ and $\{J_t \in \cJ(M,\om)\}_{t\in \rS^1}$ a generic loop of $\om$-compatible almost complex structures, denote 
\begin{equation*}
CF_k(H;J)^{<a}=\bigg\{\sum \la_{\ol{x}}\cdot\ol{x}\,\,\big|\,\,\ol{x}\in\til{\cO}(H),\,\text{CZ}(\ol{x})=k,\, \la_{\ol{x}}\in\bK,\,\cA_H(\ol{x}) < a\bigg\}, 
\end{equation*}
where $\#\{\la_{\ol{x}}\neq0\ |\,\cA_{H}(\ol{x})>c\}<\infty$ for every $c\in\R$, the vector space over $\bK$ generated by critical points of the Hamiltonian action functional of filtration level $<a$. The graded $\bK$-vector space $CF_*(H,J)^{<a}$ is endowed with the Floer differential $d_{H;J}$, which is defined as the signed count of isolated solutions (quotiented out by the $\R$-action) of the asymptotic boundary value problem on maps $u:\R\times\rS^1\ra M$ defined by the negative gradient of $\cA_H$ \cite{Salamon-lectures,SalamonZehnder}. In other words, the boundary operator counts the finite energy solutions to the Floer equation
\begin{equation*}
\frac{\del u}{\del s} + J_t(u)\frac{\del u}{\del t} + \nabla H_t(u) = 0\quad\text{s.t.}\quad E(u) = \int_\R\int_{\rS^1}\bigg\|\frac{\del u}{\del s}\bigg\|^2 dt\,ds<\infty
\end{equation*}
that converge to periodic orbits $x_-$ and $x_+$ as $s$ tends to $\pm\infty$ such that the capping $\ol{x}_-\#u$ is equivalent to $\ol{x}_+$ and $\text{CZ}(\ol{x}_-)-\text{CZ}(\ol{x}_+)=1$. In this case the Floer trajectory $u$ satisfies $E(u)=\cA_H(\ol{x}_-) - \cA_H(\ol{x}_+)$.
We thus obtain the \textit{filtered Floer chain complex} $(CF_*(H;J)^{<a}, d_{H;J})$ which is a sub-complex of the  \textit{total Floer chain complex} obtained in the same way by setting $a=+\infty$. Furthermore, for an interval $I=(a,b)$, $a<b$, $a,b\in\R\setminus$Spec$(H)$ we define the Floer complex in the action window $I$ as the quotient complex
\begin{equation*}
CF_*(H;J)^I = CF_*(H;J)^{<b}/CF_*(H;J)^{<a}.
\end{equation*}
The resulting homology of this complex $HF_*(H)^I$ is the \textit{Floer homology} of $H$ \textit{in the action window $I$} and it is independent of the generic choice of almost complex structure $J$. So the (total) Floer homology of $H$ can be obtained by setting $a=-\infty$ and $b=+\infty$. We note that $CF_*(H;J)$ is naturally a module over the Novikov field $\Lambda_{\bK} = \bK[[q^{-1},q]$ with $q$ a variable of degree $2N.$ Indeed we define $q^{-1} \cdot \ol{x} = \ol{x} \# A_0,$ for $A_0$ the generator of $\Gamma$ with $\langle c_1(TM), A_0 \rangle = N$ and extend it to a module structure in the natural way.

Observe that by interpolating between distinct Hamiltonians through generic families and writing the Floer continuation map, where the negative gradient depends on the $\R$-coordinate, one can show that $HF_*(H)$ does not depend on the Hamiltonian, and  $HF_*(H)^{I}$ depends only on the homotopy class of $\{\phi^{t}_H\}_{t\in[0,1]}$ in the universal cover $\til\Ham(M,\om)$ of the Hamiltonian group $\Ham(M,\om)$. Also, when $M$ is rational the above construction extends by a standard continuity argument to degenerate Hamiltonians. 

{\begin{rmk}\label{rmk: virtual_cycle}
Since we deal with negative monotone symplectic manifolds, it is important to emphasize that for our arguments to apply to this case in general, we must make use of the machinery of virtual cycles (see \cite{FukayaOno1,LiuTian,FO3:book-vol12} and references therein) to guarantee that the Floer differential is defined. In this case, the ground field $\bK$ should be of characteristic zero. However, if the minimal Chern number of $(M,\om)$ is $N \geq n-2,$ then $(M,\om)$ fits into the framework of semi-positive symplectic manifolds, and hence classical transversality techniques suffice (see \cite{HS-Novikov}). 
\end{rmk}}

\subsubsection{The irrational case}
In this paper we also consider the case in which the manifold $M$ is symplectically Calabi-Yau, which includes the possibility of it being irrational. In this case we have to work a little harder if $H$ is degenerate since the continuation argument above does not work as before, since non-spectral $a,b$ for $H$ do not necessarily remain non-spectral even for arbitrarily small perturbations $H_{1}$ of $H$. Moreover, the resulting homology groups depend on the choice of non-degenerate perturbation $H_{1}$. We shall follow \cite{Hein-CC} to work around this issue.

For a fixed Hamiltonian $H$ and action window $I=(a,b)$ with $a,b\in\R\setminus$Spec$(H)$, consider the set of non-degenerate perturbations $\til{H}$ whose action spectrums do not include $a$ and $b$ and $H\leq\til{H}$ (i.e. $H(t,x)\leq \til{H}(t,x)$ for all $x\in M$ and $t\in \rS^1$). Observe that $\leq$ induces a partial order in the set of perturbations. In addition, by considering a monotone decreasing homotopy $\til{H}^s$ from $\til{H}^0$ to $\til{H}^1$, one obtains an induced homomorphism between the Floer homology groups. These give rise to transition maps $HF_*({H'})^I\ra HF_*({H''})^I$ whenever $H''\leq H'$. Therefore, we can define the filtered Floer homology of $H$ by taking the direct limit \[HF_*(H)^I = \lim_{\longrightarrow} HF_*(H')^I\] over the homology groups of the perturbations satisfying the aforementioned conditions. We remark that in the case where $H$ is non-degenerate or $M$ is rational, this definition coincides with the usual filtered Floer homology groups.



\subsubsection{Local Floer homology}
In this section we shall follow \cite{GG-negmon} in order to briefly review the construction of the \textit{local Floer homology} of a Hamiltonian $H$ at a capping $\ol{x}$ of an isolated $1$-periodic orbit $x$.

Since $\ol{x}$ is isolated we can find an isolating neighbourhood $\cl{U}$ of $x$ in the extended phase-space $\rS^{1}\times M$ whose closure does not intersect the image $\{(t,y(t))\}_{t\in[0,1]}$ of any other orbit $y$ of $H$. For a $C^2$ small enough non-degenerate perturbation $H'$ the orbit $x$ splits into finitely many 1-periodic orbits $\cO(H',x)$ of $H'$ which are contained in $U$ and whose cappings are inherited from $\ol{x}$. We denote by $\cO(H',\ol{x})$ the capped $1$-periodic orbits $\ol{x}$ splits into. Moreover, we can also guarantee that any Floer trajectory and any broken trajectory, between capped orbits in $\cO(H',\ol{x})$ are contained in $U$. For a base field $\bK$ we consider the vector space $CF_*(H,\ol{x})$ generated by $\cO(H',\ol{x})$, which by the above observation naturally inherits a Floer differential and a grading by the Conley-Zehnder index. The homology of this chain complex is independent of the choice of the perturbation $H'$ once it is close enough and it is called the local Floer homology of $H$ at $\ol{x}$; it is denoted by $HF^{\loc}_*(H,\ol{x})$. This group depends only on the class $\til{\phi}$ of $\{\phi^t_H\}$ in the universal cover $\til{\Ham}(M,\om)$ and the capped orbit $\ol{x},$ in the sense that homotopic paths have choices of cappings of orbits corresponding to a fixed point $x \in \fix(\phi)$ in bijection, and the corresponding groups are canonically isomorphic. Hence we write  $HF^{\loc}_{\ast}(H,\ol{x}) = HF^{\loc}_{\ast}(\til{\phi},\ol{x}).$ If we ignore the $\Z$-grading, then the group depends only on $\phi = \phi^1_H$ and $x \in \fix(\phi).$ In this case, we write $HF^{\loc}(\phi,x)$ for the corresponding local homology group which is naturally only $\zt$-graded.



 Let $\ol{x}$ be a capped 1-periodic orbit of a Hamiltonian $H$. We define the \textit{support} of $HF^{\loc}_*(H,\ol{x})$ to be the collection of integers $k$ such that  $HF^{\loc}_k(H,\ol{x})\neq 0$. By the continuity of the mean-index and by equation (\ref{mi_inqty}) it follows that $HF^{\loc}_*(H,\ol{x})$ is supported in the interval $[\Delta(\ol{x})-n,\,\Delta(\ol{x})+n]$. One can show that if $x$ is weakly non-degenerate then it is actually supported in $(\Delta(\ol{x})-n,\,\Delta(\ol{x})+n)$. We shall explore the idea behind the proof of this second fact later as we use the same argument to prove a similar claim in slightly greater generality, namely, that of an isolated compact {path}-connected family of contractible fixed points. 

\subsection{Quantum homology and PSS isomorphism} 
In the present section we describe the quantum homology of a symplectic manifold. It might be helpful to think of it as the Hamiltoniana Floer homology in the case the Hamiltonian is given by a $C^2$-small time-independent Morse function. Alternatively, one may consider it as the cascade approach \cite{Frauenfelder} to Morse homology for the unperturbed symplectic area functional on the space $\til\cL_{pt}M$. For a more detailed inspection of these subjects we refer to \cite{LeclercqZapolsky,OhBookI,SeidelBook}.

\subsubsection{Quantum homology}
For a fixed ground field $\bK$ and the Novikov field $\Lambda_{\bK}=\bK[[q^{-1},q]$ of $(M,\om)$, where $\deg(q)=2N$, we set $QH(M)=QH(M,\bK)=H_*(M;\Lambda_{\bK})$ as a $\Lambda_{\bK}$-module. This module has the structure of a graded-commutative unital algebra over $\Lambda_{\bK}$ whose product, denoted by $*$, is defined in terms of 3-point genus 0 Gromov-Witten invatiants  \cite{Liu-assoc,McDuffSalamon-BIG,RuanTian-qh1,RuanTian-qh2,Witten-2d}. It can be thought of as a deformation of the usual intersection product on homology. As in the classical homology algebra, the unit for this quantum product is the fundamental class $[M]$ of $M$.

\subsubsection{Piunikhin-Salamon-Schwarz isomorphism}
Under our conventions for the Conley-Zehnder index, one obtains a map \begin{equation*}
PSS:QH_{*}(M)\ra HF_{*-n}(H),
\end{equation*}
by counting (for generic auxiliary data) isolated configurations of negative gradient trajectories $\gamma:(-\infty,0]\ra M$ incident at $\gamma(0)$ with the asymptotic $\displaystyle\lim_{s \to -\infty}u(s,\cdot)$ of a map $u:\R\times\rS^1\ra M$ of finite energy, satisfying the Floer equation
\begin{equation*}
\frac{\partial u}{\partial s} + J_t(u)\left(\frac{\partial u}{\partial t} - X_k^t(u)\right)=0,
\end{equation*}
for $(s,t)\in\R\times\rS^1$ and $K(s,t)\in C^\infty(M,\R)$ a small perturbation $\beta(s)H_t$, such that it coincides with $H$ for $s\ll-1$ and $s\gg+1$. Also, we ask that $\beta:\R\ra[0,1]$ is a smooth function satisfying $\beta(s)=0$ for $s\ll-1$ and $s\gg +1$. This map produces an isomorphism of $\Lambda_{\bK}$-modules, which intertwines the quantum product on $QH(M)$ with the pair of pants product on Hamiltonian Floer homology. This map is called the \textit{Piunikhin-Salamon-Schwarz isomorphism}. 


\subsection{Spectral invariants in Floer theory}\label{sec:action_carrier}
We review the theory of spectral invariants following the works of \cite{PolterovichRosen,GG-negmon,OhBookI} which contain a more exhaustive list of properties and finer details of the construction. 

Let $(M,\om)$ be a closed symplectic manifold, $H$ a generic Hamiltonian and $\{J_t\}_{t\in[0,1]}$ a loop of $\om$-compatible almost complex structures. For $a\in\R\setminus\Spec(H)$ the inclusion of the filtered Floer complex into the total complex induces a homorphism
\begin{equation*}
i_a:HF(H)^{<a}\ra HF(H).
\end{equation*}
Using the identification provided by the PSS isomorphism $QH_*(M)\cong HF_{*-n}(H)$ we then define
\begin{equation*}
c(\al_M,H) = \inf\{a\in\R\,|\,\,PSS(\al_M)\in \ima(i_a)\}.
\end{equation*}

From the definition one can see that the spectral invariants do not depend on the choice of almost complex structures. In addition, for $H\in\cH$ the spectral invariant $c(\al_M,H)$ depends only on the class $\til\phi_H$ of $\{\phi_H^t\}$ in the universal cover $\til\Ham(M,\om)$, consequently, we also denote $c(\al_M,\til\phi_H)=c(\al_M,H)$. 

\begin{df}[Non-Archimedean valuation]
Let $\Lambda$ be a field. A non-Archimedean valuation on $\Lambda$ is a function $\nu:\Lambda\ra\R\cup\{+\infty\}$, such that
\begin{enumerate}[label = (\roman*)]
    \item $\nu(x)=+\infty$ if and only if $x=0$,
    \item $\nu(xy)=\nu(x)+\nu(y)$ for all $x,y\in\Lambda$,
    \item $\nu(x+y)\geq\min\{\nu(x),\nu(y)\}$ for all $x,y\in\Lambda$.
\end{enumerate}
The Novikov field $\Lambda_{\bK}=\bK[[q^{-1},q]$ possesses a non-Archimedean valuation \\$\nu:\Lambda_{\bK}\ra\R\cup\{+\infty\}$ given by setting $\nu(0)=+\infty$ and \begin{equation}\label{eq:valuation}
   \displaystyle \nu(\sum a_{j}q^j)=-\max\{\,j\,|\,a_{j}\neq 0\}.
\end{equation} \end{df}
Spectral invariants enjoy a wealth of useful properties established by Schwarz \cite{Schwarz:action-spectrum}, Viterbo \cite{Viterbo-specGF}, Oh \cite{McDuffSalamon-BIG,Oh-construction,Oh-LecturesonFloer} and generalized by Usher \cite{Usher-spec,Usher-duality}, all of which hold for closed rational symplectic manifolds, using the machinery of virtual cycles as discussed in Remark \ref{rmk: virtual_cycle} if necessary. We summarize below some of the relevant properties for our purposes:
\begin{enumerate}[label = (\roman*)]
	\item\textit{spectrality}: For each $\al_M\in QH(M)\setminus\{0\}$ and $H\in\cH$,
	\begin{equation*}
	c(\al_M,\til\phi_H)\in\Spec(H).
	\end{equation*}
	\item\textit{stability}: For any $H,G\in\cH$ and $\al_M\in QH(M)\setminus\{0\}$,
	\begin{equation*}
	\int_0^1\min_M(H_t-G_t)dt\,\leq\, c(\al_M,\til\phi_H)-c(\al_M,\til\phi_G)\,\leq\,\int_0^1\max_M(H_t-G_t)dt.
	\end{equation*}
	\item\textit{triangle inequality:} For any $H,G\in\cH$ and $\al_M, \al_M'\in QH(M)\setminus\{0\}$,
	\begin{equation*}
	c(\al_M*\al_M',\til\phi_H\til\phi_G)\,\leq\,c(\al_M,\til\phi_H)+c(\al_M',\til\phi_G).
	\end{equation*}
	\item\textit{value at identity:} $c(\al_{M},\til{id})=-\rho\cdot\nu(\al_{M})$ for every $\al_{M}\in QH(M)\setminus \{0\}$, where $\rho$ is the rationality constant of $(M,\om)$ and $\nu$ is as in (\ref{eq:valuation}).
	\item\textit{Novikov action:} For all $H \in \cl{H},$ $\alpha_M \in QH(M),$ $\la \in \Lambda_{\bK},$ \[c(\la \alpha_M, H) = c(\alpha_M, H) - \rho \cdot \nu(\lambda).\]
	\item\textit{non-Archimedean property:} For all $\alpha_M, \alpha'_M \in QH(M),$ \[c(\alpha_M + \alpha'_M,H) \leq \max \{c(\alpha_M,H), c(\alpha'_M,H)\}.\]
\end{enumerate}
We remark that by the continuity property, the spectral invariants are defined for all $H\in\cH$ and all the above listed properties apply in this generality. Further, we observe that for $\al_M\in QH(M)$ satisfying $\al_M*\al_M=\al_M$ the triangle inequality for the spectral invariants implies
\begin{equation*}
c(\al_M,\til\phi_{H^{(k)}})=c(\al_M,\til\phi_H^k)\leq k\cdot c(\al_M,\til\phi_H).
\end{equation*}

\subsubsection{Spectral norm} For a Hamiltonian $H\in\cH$, we define its spectral pseudo-norm by
\begin{equation}
\gamma(H)= c([M],\til\phi_{H}) + c([M],\til\phi_{\ol{H}}),
\end{equation}
where $\ol{H}$ is the Hamiltonian function $\ol{H}(t,x)= -H(1-t,x)$. A result of \cite{Oh-specnorm,Schwarz:action-spectrum,Viterbo-specGF} shows that 
\begin{equation*}
\gamma(\phi) = \inf_{\phi_H^1 = \phi}\gamma(H),
\end{equation*}
defines a non-degenerate norm $\gamma:\Ham(M,\om)\ra\R_{\geq0}$ and yields a bi-invariant distance $\gamma(\phi,\phi')=\gamma(\phi'\phi^{-1})$. We call $\gamma(\phi)$ the \textit{spectral norm} of $\phi$ and $\gamma(\phi,\phi')$ the \textit{spectral distance} between $\phi$ and $\phi'$.

\subsubsection{Carrier of the spectral invariant} In this section we review the definition of \textit{carriers of the spectral invariant} mainly following \cite{GG-negmon} and references therein. We observe that while we are going to introduce the notion of carriers specifically for the fundamental class $[M]\in QH(M)$, it can be done so for any non trivial quantum homology class $\mu$.

First, we fix $\al_M=[M]$ and denote $c(H)=c(\til\phi_H)=c([M],\til\phi_H)$. Observe that in the case of a non-degenerate Hamiltonian $H$, we have 
\begin{equation*}
c(\til\phi_H) = \inf\{\cA_H(\sigma)\,|\,\sigma\in CF_n(H),\,PSS([M]) = [\sigma]\}, 
\end{equation*}
where $\A_{H}(\sigma)$ is the maximum action of a capped orbit $\ol{x}$ entering $\sigma\in CF_{n}(H)$. By the spectrality property of spectral invariants, the infimum is obtained. Consequently, there exists a cycle $\sigma$, satisfying $[\sigma]=PSS([M])$, such that $\cA_H(\ol{x})=c(\til\phi_H)$ for an orbit $\ol{x}$ entering $\sigma$. We call $\ol{x}$ the \textit{carrier} of the spectral invariant and observe that in order to guarantee its uniqueness all the $1$-periodic orbits of $H$ need to have distinct action values. In order to generalize the notion of carriers to the case where $H$ is degenerate and has isolated orbits we first recall that for each $C^2$-small non-degenerate perturbation $H'$ every capped $1$-periodic orbit $\ol{x}$ splits into several non-degenerate $1$-periodic orbits $\cO(H',\ol{x})$ with their capping inherited from $\ol{x}$. 

\begin{df}[Carrier of degenerate $H$ with isolated orbits]\label{df:carrier_degenerate} A capped $1$-periodic orbit $\ol{x}$ is said to be a carrier of the spectral invariant if there exists a sequence $\{H'_k\}$ of non-degenerate perturbations $C^2$-converging to $H$ such that for each $k$, one of the orbits in $\cO(H'_k,\ol{x})$ is a carrier for $H'_k$. A uncapped orbit is said to be a carrier if it becomes one for a suitable capping.
\end{df}

As in the non-degenerate case, the uniqueness of the carrier follows from all the 1-periodic orbits having distinct action values. In this case, the carrier becomes independent of the choice of sequence $\{H'_k\}$. Moreover, since the Floer complex is stable under small perturbations of the Hamiltonian function and auxiliary data, we may suppose that all 1-periodic orbits have distinct action values and that $c(\til\phi_H)=c(\til\phi_{H'_k})$ for all $k$. In addition, the definition of a carrier and the continuity of the action functional and of the mean-index, readily yield
\begin{equation*}
c(\til\phi_H) = \cA_H(\ol{x})\quad\text{and}\quad 0\leq\Delta(\til\phi_H,\ol{x})\leq2n,
\end{equation*}

where the inequalities can be made strict in the case where the orbit $x$ is weakly non-degenerate. In \cite{GG-negmon} the following result was obtained.

\begin{lma}\label{lma:carrier_spectral_invariant}
	Suppose $H$ only has isolated 1-periodic orbits and let $\ol{x}$ be a carrier of the spectral invariant of the fundamental class. Then $HF^{\loc}_n(H,\ol{x})\neq0$.
\end{lma}

In Section \ref{subsubsection:gen_local_FH} below, we generalize this statement to the case of isolated {path}-connected sets of periodic orbits, and also to arbitrary quantum homology classes.

\section{Isolated connected sets of periodic points}

\subsection{Generalized perfect Hamiltonians}
Recall that a Hamiltonian $H$ is called perfect if it has a finite number of contractible periodic points of all periods. We consider the more general condition where $H$ has finitely many isolated path-connected families of 1-periodic orbits, which in turn implies that $\fix(\phi_H)$ is composed of finitely many isolated path-connected sets. 

\begin{df}\label{df:generalized_perfect}
	A Hamiltonian diffeomorphism $\phi\in\Ham(M,\om)$ is \textit{generalized perfect} whenever the following conditions are met:
	\begin{enumerate}[leftmargin=*, label = (\roman*)]
		
		\item $\text{Fix}(\phi)$ has finitely many isolated compact path-connected components.
		\item There exists a sequence of integers $k_i\ra\infty$ such that $\text{Fix}(\phi^{k_i}) = \text{Fix}(\phi)$ for all $i$.
		\item For each isolated path-connected component $\cl{F}$ of $\fix(\phi),$ the mean index $\Delta(H^{(k_i)},\ol{x}^{(k_i)})$ where $x \in \cl{F}$ and $\ol{x} \in \ol{\mf{F}}$ is a constant function of $x \in \cl{F}.$ We denote this constant by $\Delta(H^{(k_i)},\ol{\mf F}^{(k_{i})}).$
	\end{enumerate}
	
\end{df}

An isolated path-connected component $\cF\subset\fix(\phi)$ can be thought of as, and indeed called in this paper, a {\em generalized fixed point}. In this section we slightly generalize some of the theory discussed in Section \ref{sec: prelims} allowing us to treat generalized perfect Hamiltonians. We observe that the third condition in Definition \ref{df:generalized_perfect} is non-vacuous: indeed, one can construct an example of a generalized fixed point $\cF$ where the mean-index is not a constant function of $x\in\cF$ by means of the Hamiltonian suspension construction \cite[Section 3.1]{P-book} applied to an appropriate contractible Hamiltonian loop of $S^2$. However, as stated in Theorem \ref{thm: torsion is perfect}, a $p$-torsion Hamiltonian diffeomorphism is weakly non-degenerate generalized perfect: in particular, the mean-index is constant on each generalized fixed point.


\subsubsection{Lifts of generalized $1$-periodic orbits} 
Let $(M,\om)$ be a symplectic manifold in one of the three classes we are considering and $H$ a Hamiltonian function generating a generalized perfect Hamiltonian $\phi_H$ on $M$. Denote the path-connected isolated sets of $\fix(\phi_H)$ by $\cF_1,\dots,\cF_m$. For each $x\in\cF_j$ and every $j$ there is a corresponding contractible loop given by $x(t)=\phi_H^t(x)$, thus to each isolated fixed point set $\cF_j$ we can associate a subset $\fF_j$ of the space $\cL_{pt}M$ of all contractible loops in $M$. It is natural to ask whether the generalized orbits $\fF_j$ lift to the $\Gamma$-cover $\til{\cL}_{pt}M$ in a suitable manner, namely, if the preimage under the projection $Pr:\til\cL_{pt}M\ra\cL_{pt}M$ is composed of isolated path-connected ``copies" of $\fF_j$. We show that the lift exists and denote by $\ol\fF_j$ a particular lift of $\fF_j$. This is analogous to a capping of an orbit in the case of a usual Hamiltonian.

Consider the set $\fF$ associated to $\cF\in\pi_{0}(\fix(\phi_{H}))$ and let $i:\fF\ra\cL_{pt}M$ be the natural inclusion map. Formally, we are asking when, given a loop $x_0\in\fF$ and $\ol{x}_0\in Pr^{-1}(\{x_0\})$, does a lift of $i$ exists, namely, a unique map $f:\fF\ra\til\cL_{pt}M$ such that $f(x_0)=\ol{x}_0$ and $Pr\circ f = i$. From the theory of covering spaces, the existence of the lift is equivalent to $i_*(\pi_1(\fF,x_0))\subset Pr_*(\pi_1(\til\cL_{pt}M,\ol{x}_0))$.

\begin{prop}\label{existence_of_lifts}
	Let $(M,\om)$ be a symplectic manifold in one of the three classes considered in this paper and $\phi_H$ a generalized perfect Hamiltonian diffeomorphism. Then each generalized orbit $\fF$ can be lifted to $\ol{\fF}$ in a unique manner specified by a loop $x_0\in\fF$ and an element in its fiber $\ol{x}_0\in Pr^{-1}(x_0)$. 
\end{prop}
\begin{proof}
Let $\ga$ be a loop in $\fF$ such that $\ga_0=x_0$. We show that we can find $\til\ga$ such that $i\circ\ga=Pr\circ\til\ga$, which implies the claim of the theorem. We build $\til\ga$ in a natural way by defining the capping at $\ga_s$ to be given by gluing the ``cylinder" given by traversing the loop $\ga$ from $0$ to $s$ to the capping $\ol{x}_0$. To see that the capped orbits $\til\ga_0$ and $\til\ga_1$ are equivalent in $\til\cL_{pt}M$, we show that
\begin{equation}\label{symplectic_area_torus}
\int_{\rT^2}\ga^*\om = 0 
\end{equation}
for every loop $\ga$ in $\fF$. We can then guarantee the existence of a lift. Equation (\ref{symplectic_area_torus}) follows from the continuity of $\cA_H$ and the fact that $\Spec(H)$ has zero measure in $\R$. Indeed, $\cA_H(\til\ga_s)=\cA_H(\til\ga_0)$ for every $s$, otherwise, the fact that $\til\ga_s$ is a critical point for each $s$ would imply that $\cA_H(\bigcup_{0\leq t\leq s}\til\ga_t)$ is a positive measure subset of $\Spec(H)$. Finally $\cA_H(\til\ga_1)=\cA_H(\til\ga_0)$ amounts to fulfilling the sufficient condition given by equation (\ref{symplectic_area_torus}).
\end{proof}

	

\subsubsection{Generalized local Floer homology}\label{subsubsection:gen_local_FH}
In this section, we define a version of local Floer homology for a generalized capped orbit $\ol\fF\subset\til\cL_{pt}M$ of a $1$-periodic Hamiltonian $H$ in a way closely related to what was done in \cite{McLean-geodesics, GG-negmon}. The differences being that we are beyond the aspherical case and we are dealing with path-connected components of $\fix(\phi_{H})$ instead of isolated points. The proofs of \cite{McLean-geodesics} used to define the local homology are valid in this case with nearly no modifications. The notion of local Floer homology in a more general setting goes back to the original work of Floer \cite{Floer-MorseWitten,Floer3} and has been revisted a number of times, for example in the work of Pozniak \cite{Pozniak}. The main ingredients of the construction are as follows.


For each $\cF\in\pi_{0}(\fix(\phi_{H}))$, we can find an isolating neighbourhood $U_{\cF}$ of the corresponding generalized $1$-periodic orbit $\mf{F}$ in the extended phase-space $\rS^{1}\times M$. That is 
\begin{equation*}
    \{ (t,\phi^{t}_{H}(x))\;|\; {t\in[0,1],x\in\cF} \} \subset U_{\cF}
\end{equation*}
and $U_{\cF}$ is disjoint from $U_{\cF'}$ for any pair of distinct generalized fixed points $\cF, \cF'$. Such an open set $U_{\cF}$ in the extended phase-space can be constructed, using the isotopy $\phi^{t}_{H},$ from an open neighbourhood of $\cF$ in $M.$ Hence by a slight abuse of notation we think of $U_{\cF}$ as a neighborhood of $\cF$ in $M.$

Then there exists an $\varepsilon>0$ small enough such that for any non-degenerate Hamiltonian perturbation $H'$ satisfying $\|H-H'\|_{C^2}<\varepsilon$ the orbits which $\fF$ splits into $\cO(H',\fF)$ will be contained in $U_{\cF}$ and so will every (broken) Floer trajectory connecting any such two orbits (see Lemma \ref{crossing_energy}). We can now consider the complex $CF_*(H',\ol\fF)$ over a ground field $\bK$ generated by the capped $1$-periodic orbits $\cO(H',\ol{\fF})$ which $\ol{\fF}$ splits into, where the cappings are naturally produced from the specific lift $\ol\fF$. One can see that this complex will be graded by the Conley-Zehnder index and will have a well defined differential. By a standard continuation argument one can show that the homology of this complex is independent of the non-degenerate perturbation (once it is close enough) and of the choice of almost complex structure. We refer to the resulting homology as the local Floer homology of $H$ at $\ol\fF$ and denote it by $HF^{\loc}_*(H,\ol\fF)$. Set \[ \Delta^{\min}(H,\mf{F}) = \min_{\ol{x}\in\ol{\fF}}\Delta(H,\ol{x}),\] \[ \Delta^{\max}(H,\mf{F}) = \max_{\ol{x}\in\ol{\fF}}\Delta(H,\ol{x}) \] for the minimimum and maximum of the mean-index $\Delta(H,\ol{x})$ for $\ol{x} \in \ol{\mf{F}}.$

We claim that if $\fF$ is a family of weakly non-degenerate orbits, the support of $HF^{\loc}_*(H,\ol\fF)$ satisfies
\begin{equation}\label{eq:supp_local_FH_generalized}
    \text{Supp}(HF^{\loc}_*(H,\ol\fF)) \subset \left(\Delta^{\min}(H,\ol{\fF})-n, \Delta^{\max}(H,\ol{\fF})+n\right).
\end{equation}
In fact, with a simple argument following from the continuity of the mean-index and inequality (\ref{mi_inqty}) one obtains that $\text{Supp}(HF^{\loc}_*(H,\ol\fF))$ satisfies the non strict version of (\ref{eq:supp_local_FH_generalized}) without having to impose the weak non-degeneracy condition. In order to obtain the strict inequalities we use the weakly non-degenerate assumption and the compactness of $\cF$ to argue as in \cite{SalamonZehnder}. In the situation where the Hamiltonian is generalized perfect, we obtain the following.

\begin{lma}\label{lma:support_localFH_gen_perf}
Suppose $H$ is a generalized perfect Hamiltonian and let $\ol{\fF}$ be a generalized capped orbit of $H$. Then $HF^{\loc}_*(H,\ol\fF)$ is supported in the open interval $(\Delta(H,\ol{\fF})-n,\Delta(H,\ol{\fF})+n)$.
\end{lma}

Furthermore, the notion of action carriers discussed in section \ref{sec:action_carrier} remains valid in this generalized setting by altering isolated fixed points to generalized fixed points in Definition \ref{df:carrier_degenerate}. Thus, the spectral invariant $c([M],H)$ is carried by a capped generalized periodic orbit $\ol\fF$ of $H$. In this case, we have the following generalization of Lemma $\ref{lma:carrier_spectral_invariant}$, whose proof, once Lemma \ref{crossing_energy} below is taken into account, follows just as in \cite{GG-negmon}.
\begin{lma}\label{lma:carrier_spectral_invariant_genralized}
	Suppose $H$ has only a finite number of generalized fixed points and let $\ol{\fF}$ be a carrier of the spectral invariant of the fundamental class. Then $HF^{\loc}_n(H,\ol{\fF})\neq0$.
\end{lma}


\begin{rmk}\label{rmk:generalized_localFH_recapping}
Let $\cF\in\pi_{0}(\fix(\phi))$ and $\fF\subset\cL_{pt}M$ the associated generalized one-periodic orbit. We remark that different choices of lifts $\ol{\fF}$ result in isomorphic local Floer homology groups with a shift in index given by an integer multiple of $2N$. In particular, if $A\in\Gamma$, then
\begin{equation*}
    HF^{\loc}_{*}(H,\ol{\fF}\#A)\cong HF^{\loc}_{*+2\brat{c_{1}(TM),A}}(H,\ol{\fF}),
\end{equation*}
where, $\ol{\fF}\#A$ denotes the unique choice of lift containing the capped orbit $\ol{x}\#A$, for $x\in\cF$ and $\ol{x}\in\ol{\fF}$. From this discussion, we conclude that $\dim_{\bK}HF^{\loc}_{*}(H,\ol{\fF})$ does not depend on the capping of $\fF$, hence, the notation $\dim_{\bK}HF^{\loc}_{*}(H,\cF)$ is justified in this case. Furthermore, when $(M,\om)$ is symplectically Calabi-Yau the local Floer homology does not depend on the choice of lift, thus we denote it by $HF^{\loc}_{*}(H,\cF)$. This is analogous to the effect of recapping on local Floer homology in the case of isolated fixed points. 
\end{rmk}

We shall require a slightly more general statement regarding carriers of quantum homology classes. The definition of a carrier $\ol{\mf{F}}$ of a quantum homology class $\alpha_M \in QH(M)$ is the same as for the fundamental class, with $[M]$ replaced by $\alpha_M.$ We then have the following result.

\begin{lma}\label{lma:carrier_spectral_invariant_genralized-QH}
Let $\alpha_M \in QH_k(M)\setminus \{0\}$ be a homogeneous element of degree $k.$ Suppose $H$ has only finitely many (contractible) generalized fixed points and let $\ol{\fF}$ be a carrier of the spectral invariant of $\alpha_M.$ Then $HF^{\loc}_k(H,\ol{\fF})\neq0$.
\end{lma}

In fact a stronger result is true, of which this statement is a direct consequence. It was proven as \cite[Theorem D]{S-PRQS} in the context of $\phi^1_H$ with isolated fixed points, but its proof adapts essentially immediately to the context of a finite number of (contractible) generalized fixed points, once Lemma \ref{crossing_energy} is established. We recall that its proof relies on homological perturbation techniques, starting from the decomposition of Section \ref{subsubsec: decomposition} and constitutes a Novikov-field version of the canonical $\Lambda^0$-complexes from \cite{S-HZ}. Our case differs from the one in \cite{S-PRQS} by replacing fixed points by generalized fixed points everywhere.

\begin{thm}\label{thm: homological perturbation complex}
Let $(M,\om)$ be a closed rational symplectic manifold. Consider the class $\til{\phi} \in \til{\Ham}(M,\om)$ of the Hamiltonian flow $\{\phi^t_H\}_{t \in [0,1]}$ of $H\in \cl H,$ with $\fix(\phi^1_H)$ consisting of a finite number of generalized fixed points. For a ground field $\bK,$ there is a homotopy-canonical complex $(C(H), d_H)$ over the Novikov field $\Lambda_{\bK}$  on the action-completion of \[\oplus HF^{\loc}_{\ast}(\til{\phi},\ol{\mf{F}})\] the sum running over all capped generalized one-periodic orbits $\ol{\mf{F}} \in \til{\cO}(H),$ that is free and graded over $\Lambda_{\bK},$ and is strict, that is $\cl{A}_H (d_{H} (y)) < \cl{A}(y)$ for all $y \in C(H),$ with respect to the non-Archimedean action-filtration $\cl{A}_H$ on $C(H)$ defined as follows: \begin{equation}\label{eq: action of sum loc} \cl A_H( \sum \lambda_j y_j ) = \max\{ -\nu(\lambda_j) + \cl A_H(y_j) \},\end{equation} \[\cl{A}_H(y_j) = \cl{A}_H(\ol{\mf{F}}_{i(j)})\] for a $\Lambda$-basis $\{y_j\}$ of $C(H)$ determined by $\{y_j \,|\, i(j) = i\}$ being a basis of $HF^{\loc}_*(\til{\phi},\ol{\mf{F}}_{i}),$ where $\fix(\phi) = \{\cl{F}_i\}$ and for each $i,$ $\ol{\mf{F}}_i$ is a choice of a lift of the generalized one-periodic orbit $\mf{F}_i$ corresponding to $\cl{F}_i$ to a capped generalized periodic orbit in $\til{\cl{O}}(H).$  Furthermore the filtered homology $HF(H)^{<a}$ is given by $HF(C(H)^{<a}),$ $C(H)^{<a} = (\cl{A}_H)^{-1}\,(-\infty,a),$ for all $a \in \R \setminus \spec(H).$ In particular $HF(H) = H(C(H),d_{H}) \cong QH(M;\Lambda_{\bK}).$ Moreover, for all $a\leq b,$ $a,b \in (\R\setminus \Spec(H)) \cup \{\infty\},$ the comparison map $HF(H)^{<a} \to HF(H)^{<b}$ is induced by the inclusion $C(H)^{<a} \to C(H)^{<b}.$
\end{thm}

\begin{df}[Visible spectrum]
We define the visible spectrum of a Hamiltonian function $H$ as 
\begin{equation*}
    \spec^{vis}(H)=\{\cA_{H}(\ol{\fF})\,|\,HF^{\loc}_{*}(H,\ol{\fF})\neq 0\},
\end{equation*}
where $\cA_{H}(\ol{\fF})$ denotes the action of any capped orbit $\ol{x}\in\ol{\fF}$ for a lift $\ol{\fF}$ associated to a generalized fixed point $\cF\subset\fix(\phi_{H})$. Indeed, an argument similar to the proof of Proposition \ref{existence_of_lifts} shows that the restriction $\cA_{H}|_{\ol{\fF}}$ is constant. It is clear that $\spec^{vis}(H)\subset\spec(H)$. In the context of barcodes (see Section \ref{subsubsec: barcodes}), the visible spectrum corresponds to the endpoints of all bars of the barcode $\cl{B}(H)$ associated to the filtered Floer homology of $H.$
\end{df}

\subsubsection{Crossing Energy}
We show that for a small perturbation $H'$ of a generalized perfect Hamiltonian $H$, every Floer trajectory $u$ connecting orbits of $H'$ contained in distinct isolating neighbourhoods has energy bounded below by a constant independent of the perturbation. This is an important technical step. 

\begin{lma}\label{crossing_energy}
	There exist $\delta>0$ and $\varepsilon>0$ such that for every non-degenerate perturbation $H'$ of $H$ satisfying $\|H-H'\|_{C^2}<\varepsilon$, every orbit in $\cO(H',\fF_j)$ is contained in $U_{\cF_j}$ for $j=1,..,m$ and every Floer trajectory $u$ connecting capped orbits in distinct isolating neighborhoods satisfies $E(u)>\delta$.
\end{lma}
\begin{proof}
Suppose there exists a sequence of non-degenerate Hamiltonians $\{H'_{k}\}$ $C^2$-convering to $H$ and a sequence of Floer trajectories $u_{k}$ of $H'_{k}$ connecting orbits in distinct isolating neighbourhoods such that $E(u_{k})\ra 0$. Since $H$ has finitely many generalized fixed points, we may suppose without loss of generality that all the Floer trajectories $u_{k}$ connect orbits in $U_{\cF}$ to orbits in $U_{\cF'}$, where $\cF,\cF'\in\pi_0(\fix(\phi_{H}))$.


By a compactness result of \cite{Fish-compactness} and arguing as in \cite{McLean-geodesics} we obtain the existence of a Floer trajectory $u$ of $H$ connecting an orbit in $U_{\cF}$ to an orbit in $U_{\cF'}$ such that $E(u) = 0$. Thus  
\begin{equation*}
\frac{\partial u}{\partial s} = 0 \qquad \text{and} \qquad \frac{\partial u}{\partial t} = X_{H_t}
\end{equation*}
\noindent which in turn, implies that for each $s$, the loop $u_s=u(s,.)$ is a 1-periodic orbit of $H$. This contradicts the fact that the generalized fixed points of $H$ are isolated.
\end{proof}

\subsubsection{Decomposition of Floer differential}\label{subsubsec: decomposition}
An important feature related to local Floer homology concerns the decomposition of the full differential defined on the complex $CF_*(H')$ in to the sum of local differentials of complexes $CF_*(H,\ol\fF_j)$ -- for all the different lifts of the finitely many critical sets --  and in to an additional component we shall call $D$. Note that here $H'$ is a close enough non-degenerate Hamiltonian in the aforementioned sense. We mean that for a chain $\sigma\in CF_*(H')$ we have 
\begin{equation}
\partial\sigma = \sum\til\partial_j\sigma + D\sigma
\end{equation}
where, $\til\partial_j$ represents an extension of the local differential of the complex $CF_*(H,\ol\fF_j)$ obtained by setting $\til\partial_j\bar{x}=0$ for every capped orbit which does not belong $\cO(H',\ol{\fF_j})$. Loosely speaking, $D$ only ``counts" Floer trajectories connecting orbits contained in disjoint isolating open sets $U_{\cF_j}$.

\begin{rmk}
	Suppose $\sigma$ is a chain in the complex $CF_*(H')$ and $\Bar{z}$ is an orbit entering $D\sigma$. Naturally, there exists $0\leq k \leq m$ such that $\Bar{z}\in CF_*(H,\ol\fF_k)$ for a particular lift of $\fF_k$ and a Floer trajectory $u$ connecting an orbit $\Bar{y}\in CF_*(H,\ol\fF_l)$ to $\bar{z}$ for $l\neq k$ (and a particular lift of $\fF_l$). We then obtain 
\begin{equation}
\cA_{H'}(\bar{z}) =\cA_{H'}(\bar{y}) - E(u) <  \cA_H'(\bar{y}) - \delta
\end{equation}
where the first equality comes from the fact that the energy of a Floer trajectory connecting two capped orbits is equal to their action difference and the $\delta$ comes from the uniform lower bound for the crossing energy from Lemma \ref{crossing_energy}. In other words \[ \cA_{H'}(D x) < \cA_{H'}(x) - \delta\] for all $x \neq 0$ in $CF_*(H').$ 
\end{rmk}

\subsubsection{Barcodes of Hamiltonian diffeomorphisms}\label{subsubsec: barcodes}

The proof of Theorem \ref{thm: torsion is PR} uses notions and results regarding barcodes of Hamiltonian diffeomorphisms, in the case where they have a finite number of contractible generalized fixed points. Hitherto, this theory was developed mostly for the case where the generalized fixed points are in fact points, yet given Lemma \ref{crossing_energy}, all relevant results generalize to our situation. In the next section we describe the main Smith-type inequality regarding the behavior of barcodes under iteration.

We summarize the properties necessary for us as follows, and refer to \cite{PolShe,PolSheSto, UsherZhang, KS-bounds, S-Zoll, S-HZ} for further discussion of this notion, in the context of continuity in the Hofer distance and the spectral distance in particular. For convenience, we work in the setting of monotone symplectic manifolds, yet natural analogues of various statements exist in the semi-positive, rational, and general settings.

\begin{prop}\label{prop: barcodes}
Let $(M,\om)$ be a monotone symplectic manifold with $\cl{P}_{\om} = \rho \cdot \Z,$ $\phi \in \Ham(M,\om)$ with $\fix(\phi)$ consisting of a finite number of generalized fixed points. Let $\bK$ be a coefficient field. Let $H$ be a Hamiltonian generating $\phi.$ Then $\Spec(H) \subset \R$ is a discrete subset, and there exists a countable collection \[\cl{B}(H) = \cl{B}(H; \bK) = \{(I_i,m_i)\}_{i \in \cl{I}}\] of intervals $I_i$ in $\R$ of the form $I_i = (a_i,b_i]$ or $I_i = (a_i,\infty),$ called {\bf bars} with multiplicities $m_i \in \Z_{>0}$ such that the following properties hold:

\begin{enumerate}[label = (\roman*)]
\item The group $\rho \cdot \Z$ acts on $\cl{B}(H)$ in the sense that for all $k \in \Z$ and all $(I,m) \in \cl{B}$ we have $(I+\rho k, m) \in \cl{B}.$
\item For each window $J = (a,b)$ in $\R,$ with $a,b \notin \spec(H),$ only a finite number of intervals $I$ with $(I,m) \in \cl{B}$ have endpoints in $J.$ Furthermore, \[ \dim_{\bK} HF(H)^{J} = \displaystyle\sum_{(I,m) \in \cl{B}(H),\; \# \partial I \cap J = 1} m,\] where for an interval $I = (a,b],$ $\partial I = \{a,b\},$ and for $I = (a, \infty),$ $\partial I = \{a\}.$
\item In particular for $a \in \Spec(H),$ and $\eps > 0$ sufficiently small, so that we have $(a-\eps, a+\eps) \cap \Spec(H) = \{a\},$ \[ \dim_{\bK} HF(H)^{(a-\eps, a+\eps)} = \displaystyle\sum_{(I,m) \in \cl{B}(H),\; a \in \partial I} m,\] \[ \dim_{\bK} HF(H)^{(a-\eps, a+\eps)} = \displaystyle \sum_{ \cl{A}(\ol{\mf{F}}) = a } \dim_{\bK} HF^{\loc}(H, \ol{\mf{F}}).\]

\item There are $K(\phi,\bK)$ orbits of finite bars counted with multiplicity, and $B(\bK)$ orbits of infinite bars counted with multiplicity, under the $\rho \cdot \Z$ action on $\cl{B}(H).$ These numbers satisfy: \[B(\bK) = \dim_{\bK} H_*(M;\bK)\] and \[ N(\phi,\bK) = 2 K(\phi,\bK) + B(\bK),\] where \[ N(\phi,\bK) = \sum \dim_{\bK} HF^{\loc}(\phi, \cl{F})\] is the {\bf homological count of the fixed points of $\phi,$} the sum running over all the set $\pi_0(\fix(\phi))$ of its generalized fixed points.

\item There are $K(\phi,\bK)$ {\bf bar-lengths} corresponding to the finite orbits, \[ 0 < \beta_1(\phi,\bK) \leq \ldots \leq \beta_{K(\phi,\bK)}(\phi,\bK),\] which depend only on $\phi.$ We call \[ \beta(\phi,\bK) = \beta_{K(\phi,\bK)}(\phi,\bK) \] the {\bf boundary-depth} of $\phi,$ and \[ \beta_{\mrm{tot}}(\phi,\bK) = \sum_{1 \leq j \leq {K(\phi,\bK)}} \beta_j(\phi,\bK)\] its {\bf total bar-length}.
\item Each spectral invariant $c(\alpha, H) \in \Spec(H)$ for $\alpha \in QH_*(M) \setminus \{0\}$ is a starting point of an infinite bar in $\cl{B}(H),$ and each such starting point is given by a spectral invariant\footnote{In fact, representatives for the set of  orbits of infinite bars counting with multiplicity, can be obtained as spectral invariants of an orthogonal basis of $QH_*(M)$ over the Novikov field $\Lambda_{\bK},$ with respect to the non-Archimedean filtration $l_H(-) = c(-,H).$ As we shall not require this stronger statement, we refer to \cite{S-HZ, S-Zoll} for a discussion of the relevant notions.}.

\item \label{barcodes, change of H} If $H'$ is another Hamiltonian generating $\phi,$ then $\cl{B}(H') = \cl{B}(H)[c],$ for a certain constant $c \in \R,$ where $\cB(H)[c] = \{ (I_i-c,m_i)\}_{i \in \cl{I}}.$ 

\item \label{barcodes, change of coefficients} If $\mathbb{\bK}$ is a field extension of $\mathbb{F},$ and $H$ is a Hamiltonian, then $\cl{B}(H; \bK) = \cl{B}(H; \bF).$ In particular $\cl{B}(H; \bK) = \cl{B}(H; \F_p)$ if $\mrm{char}(\bK) = p,$ and $\cl{B}(H; \bK) = \cl{B}(H; \Q)$ if $\mrm{char}(\bK) = 0.$

\end{enumerate}

\end{prop}

\subsubsection{Smith theory in filtered Floer homology}

One of the main fundamental results of \cite{S-HZ} is the following Smith-type inequality, that readily adapts to our setting by Lemma \ref{crossing_energy} and its generalization to the situation of branched covers of the cylinder as in \cite[Proposition 9]{SZhao-pants}. We refer to \cite[Theorem D]{S-HZ} for a detailed argument in the case of isolated fixed points, and observe that our generalization below is formulated in such a way that the same proof applies verbatim, by replacing fixed points by generalized fixed points everywhere.

\begin{thm}\label{thm: Smith}
Let $(M,\om)$ be a monotone symplectic manifold, $p$ a prime number, and $\phi \in \Ham(M,\om)$ with $\fix(\phi),$ $\fix(\phi^p)$ each consisting of a finite number of generalized fixed points, and such that the natural inclusion $\fix(\phi) \to \fix(\phi^p)$ restricts to a homeomorphism from each generalized fixed point $\cl{F}$ of $\phi$ to a generalized fixed point of $\phi^p,$ which we denote $\cl{F}^{(p)}.$ Then \[ \beta_{\mrm{tot}}(\phi^p, \F_p) \geq p \cdot \beta_{\mrm{tot}}(\phi, \F_p).\]
\end{thm}

This inequality will be the key component in the proof of Theorem \ref{thm: torsion is PR}.

A somewhat simpler statement than Theorem \ref{thm: Smith}, is the Smith inequality in generalized local Floer homology, whose proof is precisely as in \cite{SZhao-pants} together with the crossing energy argument of Lemma \ref{crossing_energy}.

\begin{prop}\label{thm: Smith loc}
Let $(M,\om)$ be a closed symplectic manifold, $p$ a prime number, and $\phi \in \Ham(M,\om).$ Suppose that $\fix(\phi),$ $\fix(\phi^p)$ each consist of a finite number of generalized fixed points. Let $\cl{F}$ be a generalized fixed point of $\phi,$ such that the natural inclusion $\fix(\phi) \to \fix(\phi^p)$ restricted to $\cl{F}$ is a homeomorphism onto $\cl{F}^{(p)}.$ Then \[ \dim_{\F_p} HF^{\loc}(\phi,\cl{F}) \leq \dim_{\F_p} HF^{\loc}(\phi^p,\cl{F}^{(p)}).\]
\end{prop}



\subsubsection{Quantum Steenrod operations}

Quantum Steenrod operations are remarkable algebraic maps \[QSt_p: QH^*(M;\F_p) \to QH^*(M; \F_p)[[u]]\langle\theta\rangle,\] for $p$ a prime number, $u$  a formal variable of degree $2,$ and $\theta$  a formal variable of degree $1.$ As is the usual quantum product, $QSt_p$ are essentially defined by certain counts of configurations consisting of holomorphic curves in $M$ incident with negative gradient trajectories of Morse functions in $M.$ The main difference, is that $QSt_p$ uses $p$ inputs and $1$ output trajectories, and the counts are carried out in families parametrized by the classifying space $B(\Z/p\Z)$ of $\Z/p\Z.$ The investigation of the enumerative significance of these counts, in terms of various Gromov-Witten invariants, and its implications for Mirror Symmetry was started in \cite{Seidel-formal} and \cite{SeidelWilkins}.

These opertations were first proposed by Fukaya \cite{Fukaya-Steenrod}, and were formally introduced for $p=2$ by Wilkins in \cite{Wilkins}. They were then studied in \cite{Wilkins-PSS} in relation to the equivariant pair-of-pants product of Seidel \cite{Seidel-pants}. For a definition for $p>2$ odd, see \cite{Seidel-formal} and \cite{SSW}. The significance of quantum Steenrod operations in Hamiltonian dynamics was first observed in \cite{S-PRQS}, and was further investigated in \cite{CGG2},\cite{S-PRQSR}, and \cite{SSW}. While for the moment, these operations are defined in the setting of monotone symplectic manifolds, it is expected that they will be generalized to the semi-positive (also called weakly monotone) setting.

One particular property of quantum Steenrod operations that we use in this paper, which was first observed in \cite{S-PRQS} for $p=2,$ and proved in \cite{SSW} for $p>2,$ is that whenever \begin{equation}\label{eq: point class deformed} QSt_p(\mu) \neq u^{(p-1)n} \mu,\end{equation} where $\mu \in H^{2n}(M;\F_p)$ is the cohomology class Poincar\'{e} dual to the point class, the symplectic manifold $(M,\om)$ is geometrically uniruled: for each $\om$-compatible almost complex structure $J$ on $M,$ and each point $x \in M,$ there exists a $J$-holomorphic sphere $u: \C P^1 \to M,$ such that $x \in \im(u).$ Hence, we call a (monotone) symplectic manifold {\em $\F_p$ Steenrod uniruled} if condition \eqref{eq: point class deformed} holds. The algebraic significance of this condition is that $u^{(p-1)n} \mu = St_p(\mu),$ where $St_p$ is (a slightly reformulated) total Steenrod $p$-th power of the class $\mu,$ and in general \[QSt_p = St_p + O(q),\] where $O(q)$ is a collection of terms involving $q$ to power at least $1.$ These terms correspond to configurations involving $J$-holomorphic spheres of positive symplectic area, hence condition \eqref{eq: point class deformed} means that the quantum Steenrod power of the point cohomology class is {\em deformed} by holomorphic spheres.

\subsection{Floer cohomology}\label{subsec: Floer coh}

At times it shall be convenient to work with Floer cohomology and quantum cohomology of closed symplectic manifolds, instead of homology. All the preliminary results above adapt naturally to this setting. In fact, we may define \[CF^*(H,J) = CF_{n-*}(\ol{H},\ol{J}) \] where $\ol{H}(t,x) = - H(1-t,x),$ $\ol{J}_t(x) = J_{1-t}(x).$ The usual action functional $\cA_{H}$ on the left hand side, takes the form $(-\cA_{\ol{H}})$ on the right hand side. Note that hereby the cohomological differential increases the filtration, the triangle inequality for spectral invariants has the opposite inequality, and infinite bars in the barcode are of the form $(-\infty, b).$ Local Floer cohomology is defined in the same way as for homology. Action carriers, and contribution to local Floer cohomology hold similarly: $c(\mu,H)$ for $\mu \in QH^{2n}(M)$ is {carried} by a capped generalized periodic orbit $\ol{\mf{F}}$ of $H,$ if in the same sense as for homology, $\ol{\mf{F}}$ is a {\em lowest} action term in a {\em highest} minimal action representative of the image $PSS_{H}(\mu)$ of $\mu$ under the PSS isomorphism \cite{PSS} from the quantum cohomology $QH^*(M) \to HF^*(H)$ to the filtered Floer cohomology of the Hamiltonian $H.$ For $(M,\om)$ rational, in particular monotone, each non-zero class $\mu \in QH^*(M),$ and $H \in \cH$ with $\# \pi_0\fix(\phi^1_H) < \infty,$ $c(\mu,H)$ is carried by at least one generalized capped $1$-periodic orbit $\ol{\mf{F}}$ of $H.$ Furthermore, if $\mu$ is a homogeneous class of degree $k,$ and $\ol{\mf{F}}$ carries $c(\mu, H),$ then $HF^k_{\loc}(H,\ol{\mf{F}}) \neq 0.$ 

We refer to \cite{LeclercqZapolsky} for further discussion of the comparison between Floer homology and Floer cohomology.

\section{Cluster structure of the essential spectrum}

\begin{df}[Essential spectrum]
We define the essential spectrum of a Hamiltonian function $H$ as
\begin{equation*}
    \spec^{ess}(H) = \{\,c(\al,H)\,|\,\al\in QH(M)\setminus\{0\}\,\}.
\end{equation*}
Observe that spectraility property of the spectral invariants is equivalent to the inclusion $\spec^{ess}(H)\subset\spec(H)$. In fact, Lemma \ref{lma:carrier_spectral_invariant_genralized-QH} implies that $\spec^{ess}(H)\subset\spec^{vis}(H)$ for Hamiltonian diffeomorphisms with a finite number of (contractible) generalized fixed points. In the context of barcodes (see Section \ref{subsubsec: barcodes}), the essential spectrum corresponds to the endpoints of infinite bars of the barcode $\cl{B}(H)$ associated to the filtered Floer homology of $H.$
\end{df}

In what follows, we show that whenever $\gamma(H)<\rho$, the essential spectrum has a cluster structure determined by the subset produced by quantum homology classes of valuation 0. 

\begin{prop}\label{prop:cluster_struct_I}
    Suppose $M$ is a monotone symplectic manifold and $H$ a Hamiltonian function on $M$. Then,
    \begin{equation*}
        0\leq c([M],\til\phi_{H})-c(\al,\til\phi_{H})\leq \gamma(H)
    \end{equation*}
    for all $\al\in QH(M)$ such that $\nu(\al)=0,$ including all $\al\in H_*(M)\subset QH(M).$
\end{prop}
\begin{proof}
By the triangle-inequality and the value at identity properties of the spectral invariant  
\begin{equation*}
    c(\al,\til\phi_{H})=c(\al*[M],\til{\id}\,\til\phi_{H})\leq c(\al,\til\id)+c([M],\til\phi_{H})= c([M],\til\phi_{H})
\end{equation*}
for all $\al\in QH(M)$ such that $\nu(\al)=0$. In addition,
\begin{equation*}
    0=c(\al,\til\id)=c(\al*[M],\til\phi_{H}\til\phi_{\overline{H}})\leq c(\al,\til\phi_{H})+c([M],\til\phi_{\overline{H}}).
\end{equation*}
Combining both inequalities we obtain 
\begin{align*}
    0\leq c([M],\til\phi_{H})-c(\al,\til\phi_{H})\leq c([M],\til\phi_{H})+c([M],\til\phi_{\overline{H}})=\gamma(H),
\end{align*}
which concludes the proof of the proposition.
\end{proof}

\begin{prop}\label{prop:cluster_struct_II}
    Suppose $M$ is a monotone symplectic manifold with rationality constant $\rho$, $H$ a Hamiltonian function on $M$ with $\gamma(H)<\rho$ and $\al\in QH(M)$. Then, 
    \begin{equation*}
        c([M],\til\phi_{H})-\rho < c(\al,\til\phi_{H}) \leq c([M],\til\phi_{H})
    \end{equation*}
    if, and only if, $\nu(\al)=0$.
\end{prop}
\begin{proof}
If $\nu(\al)=0$, then Proposition \ref{prop:cluster_struct_I} and the hypothesis that $\ga(H)<\rho$ imply that
\begin{equation*}
        c([M],\til\phi_{H})-\rho < c(\al,\til\phi_{H}) \leq c([M],\til\phi_{H}).
\end{equation*}
Conversely, let $x_{1},\dots,x_{B}$ be a homogeneous basis of $H_*(M)\subset QH(M)$ and denote $c=c([M],\til\phi_{H})$. Then, by Proposition \ref{prop:cluster_struct_I}, we have $c(x_k,\til\phi_{H})\in (c-\rho,c]$ for all $1\leq k\leq B$. Also, for $q^{j}\in\Lambda$ the equality $c(x_{k}q^{j},\til\phi_{H})= c(x_{k},\til\phi_{H})+j\rho$ implies that $c(x_{k}q^{j},\til\phi_{H})\notin (c-\rho,c]$ for all $j\neq0$. Thus, if $c(\al,\til\phi_{H})\in (c-\rho,c]$ for
\begin{equation*}
    \al = \lambda x_k = \sum a_{j}q^{j}x_{k},
\end{equation*}

where $\lambda\in\Lambda$, the non-Archimedean property of the spectral invariant imposes that 
\begin{equation*}
    \al = a_0x_{k} + \sum_{j<0} a_{j}q^{j}x_{k},
\end{equation*}

which in turn implies $\nu(\al)=0$. In general, $\al\in QH(M)$ is of the form $\sum\lambda_{k}x_{k}$, where $\lambda_k=\sum a_{j}^{(k)}q^{j}$. Consequently, if $c(\al,\til\phi_{H})\in (c-\rho,c]$, we may argue as before to conclude

\begin{equation}\label{eq:cluster_struct_II}
    \al = \sum_{k} \bigg( a_{0}^{(k)}x_{k} + \sum_{j<0} a_{j}^{(k)}q^{j}x_{k}\bigg).
\end{equation}

Thus, $\nu(\al)=0$, which concludes the proof of the claim. 
\end{proof}

\begin{rmk}\label{rmk:cluster_struct_rational}
We would like to point out that the above propositions are valid, albeit with minor modifications to the proofs, in the more general case where $M$ is only assumed to be rational. If $M$ is negative monotone, then the base field $\bK$ is required to be of characteristic zero (see Remark \ref{rmk: virtual_cycle}).   
\end{rmk}

Let $\phi\in\Ham(M,\om)$ and suppose $\gamma(\phi)<\rho$. We can, therefore, find a Hamiltonian function $H$ generating $\phi$ such that $\gamma(H)<\rho$. Our goal is to extract information from the cluster structure of $H$ in order to bound $\gamma(\phi)$ from below. First we set notation. Put $\rS^ {1}_{\rho}=\R/\rho \cdot\Z$ and, for $a\in\R$, let $[a]\in\rS^{1}_{\rho}$ be its equivalence class. For $\theta\in\rS^{1}_{\rho}$, define 
\begin{equation*}
    \Gamma_\theta = \{(a-\rho,a]\,|\,a\in \R,\,[a]=\theta\}.
\end{equation*}

Note that the intervals in $\Gamma_\theta$ are disjoint and their union covers the real line. In addition, observe that, modulo $\rho\cdot\Z$, the set $\spec^{ess}(H)\cap \text{I}$ does not depend on the interval $\text{I}\in\Gamma_\theta$. 

\begin{df}[Spectral length]\label{df:spectral_length}
We define the $\theta$-\textit{parsed spectral length of} $H$ as 
\begin{equation*}
    l(H,\Gamma_\theta)= \diam(\spec^{ess}(H)\cap \text{I}) = \sup\{|a-b|\,\,|\,\,a,b\in\spec^{ess}(H)\cap \text{I}\,\},
\end{equation*}
where $I\in\Gamma_\theta$ is arbitrary. For $\Gamma_H = \Gamma_{[c([M],H)]}$ we call $l(H,\Gamma_{H})$ the \textit{fundamental length of} $H$. Finally, we define the \textit{spectral length} of $\phi\in\Ham(M,\om)$ as
\begin{equation}\label{eq:spec_length}
    l(\phi) = \inf\{\,l(H,\Gamma_{\theta})\,\,|\,\,\theta\in\rS^{1}_{\rho}\},
\end{equation}
where $H$ is any Hamiltonian function generating $\phi$. Note that the right-hand-side of (\ref{eq:spec_length}) does not depend on the choice of Hamiltonian. Indeed, if $H'$ is another Hamiltonian generating $\phi$, then by Proposition \ref{prop: barcodes} \ref{barcodes, change of H} $\spec^{ess}(H')= \spec^{ess}(H)+ c$ for $c\in\R$. (Another proof using Seidel elements is also possible.) 
\end{df}

\begin{rmk} We can also define an a priori larger invariant $l'(\phi) \geq l(\phi)$ of $\phi$ by $l'(\phi) = \inf_{\phi_H^1 = \phi} l(H,\Gamma_H).$ However, we find $l(\phi)$ more convenient for this paper. 
\end{rmk}


\begin{lma}\label{lma:fundamental_length}
The fundamental length of a Hamiltonian $H$ satisfies \[l(H,\Gamma_{H})\leq\gamma(H).\]If, in addition, $\gamma(H)<\rho$, then we have equality: \[l(H,\Gamma_{H}) = \gamma(H).\]
\end{lma}
\begin{proof}
Note that by definition the $\theta$-parsed spectral length of $H$ is bounded above by $\rho$ for any choice of $\theta$, in particular, $l(H,\Gamma_{H})\leq\rho$. Thus, we need only to consider the case where $\gamma(H)<\rho$. Equation (\ref{eq:cluster_struct_II}) in the proof of Proposition \ref{prop:cluster_struct_II} implies that $\#\{\spec^{ess}(H)\cap \text{I}\}<\infty$ for $\text{I}\in\Gamma_{H}$ and hence for $\text{I}\in\Gamma_{\theta}$ for any $\theta\in\rS^{1}_{\rho}$. Thus by Proposition \ref{prop:cluster_struct_II} the fundamental length of $H$ is given by  
\begin{equation*}
    l(H,\Gamma_{H})= c([M],H)-c(\al_{\min,H},H),
\end{equation*}
where, $\al_{\min,H}\in QH(M)$ has zero valuation. Consequently, Proposition \ref{prop:cluster_struct_I} implies that $l(H,\Gamma_{H})\leq\gamma(H)$. To prove equality, we observe that by the Poincar\'e duality property of spectral invariants (see \cite{OstroverAGT,EntovPolterovichCalabiQM}) and the fact that the set $\spec^{ess}(H)\cap \text{I}$ is finite, there exists $\beta\in QH(M)\setminus\{0\}$ such that $c(\beta,H)= -c([M],\ol{H})$. By adding $\gamma(H)$ to both sides of the equality we obtain $c(\beta,H)+\ga(H)=c([M],H)$, which implies 
\begin{equation*}
    c([M],H)-\rho < c(\beta,H) \leq c([M],H).
\end{equation*}
Therefore, $\ga(H)\leq l(H,\Gamma_{H})$, which gives us the claimed equality. 
\end{proof}

\begin{lma}\label{lma:spectral_length}
    Let $\phi$ be a Hamiltonian diffeomorphism. Then, $l(\phi)\leq\gamma(\phi)$.
\end{lma}
\begin{proof}
    Let $H$ be any Hamiltonian function that generates $\phi$. Observe that by definition $l(H,\Gamma_{\theta})<\rho$ for every $\theta\in\rS^{1}_{\rho}$, in particular, we have $l(\phi)<\rho$. Therefore, we may suppose $\gamma(\phi)<\rho$, in which case, we may take $H$ such that $\gamma(H)<\rho$. Consequently, Lemma \ref{lma:fundamental_length} implies that $l(H,\Gamma_{H})=\gamma(H)$, in particular we have that $l(\phi)\leq\gamma(H)$. If $H'$ is any other Hamiltonian function generating $\phi$, with $\gamma(H')\leq\gamma(H)$, the same argument implies $l(\phi)\leq\gamma(H')$. Thus, we conclude that $l(\phi)\leq\gamma(\phi)$.
\end{proof}

\begin{rmk}
Observe that Lemma \ref{lma:fundamental_length} immediately implies that if $\gamma(\phi) < \rho,$ then $l'(\phi) = \gamma(\phi).$ It is not clear that the same holds for $l(\phi),$ however we can prove that if $\gamma(\phi) < \frac{\rho}{2},$ then $l(\phi) = \gamma(\phi).$ Indeed, if $\gamma(H) < \gamma(\phi) + \eps <  \frac{\rho}{2},$ by Lemma \ref{lma:fundamental_length} we have $l(H,\Gamma_{H}) = \gamma(H) < \frac{\rho}{2}.$ However, this implies that for arbitrary $\theta \in \rS^1_{\rho},$ either $l(H,\Gamma_{\theta}) = l(H,\Gamma_{H}),$ if the partitions of $\Spec^{ess}(H)$ into clusters corresponding to $\Gamma_{H}$ and $\Gamma_{\theta}$ coincide, or $l(H,\Gamma_{\theta}) \geq \rho - l(H,\Gamma_{H}) > \frac{\rho}{2} > l(H,\Gamma_{H}),$ if they do not. Hence by taking the infima, $l(\phi) = \gamma(\phi).$ 
\end{rmk}

%
%

\begin{lma}\label{lma:newman_for_PR}
Let $\phi$ be a generalized Hamiltonian $\bK$ pseudo-rotation with sequence $k_j=j$ and take a Hamiltonian $H$ generating $\phi$. Suppose that all the distances between pairs of points in $\spec^{ess}(H)$ are rational multiples of $\rho.$ Then, there exists a positive integer $m$ such that $\gamma(\phi^m)\geq\rho.$ 
\end{lma}
\begin{proof} Fix the base coefficient field $\bK$ for all homological notions in the proof. We can suppose $\gamma(\phi)<\rho$, otherwise the implication of the theorem would be true for $m=1$. Furthermore, we note that the hypothesis of the theorem is independent of the choice of Hamiltonian function, thus, we may suppose that $\gamma(H)<\rho$, which by Lemma \ref{lma:fundamental_length}, implies $l(H,\Gamma_{H})=\gamma(H)$. Hence, we have a cluster structure determined by finitely many values of the essential spectrum of $H$ belonging to the interval
\begin{equation*}
    \text{I}_{H}=(c([M],H)-\rho\,,\,c([M],H)\,].
\end{equation*}
Thus, setting 
\begin{equation*}
    \spec^{ess}(H)\cap\text{I}_{H} = \{c_1,\dots,c_B\},
\end{equation*}
by the hypothesis of the proposition, we have
\begin{equation*}
    c_i-c_j= \frac{a_{ij}}{b_{ij}} \rho \in \rho\cdot \Q\cap[0,\rho)
\end{equation*}

for all $i\neq j$. Note that any pair of points $\al,\beta \in \Spec^{ess}(H)$ are of the form $\al=c_i+k\rho$ and $\beta=c_j+l\rho$ for integers $l$ and $k$. Thus their difference is of the form
\begin{equation}\label{eq:newman_for_PR_I}
\al-\beta=(a_{ij}/b_{ij}+(k-l))\rho.     
\end{equation}

Now, let $m$ be the integer given by $\prod_{i<j}b_{ij}$. The facts that $\fix(\phi^m)=\fix(\phi)$, $\spec^{ess}(H)=\spec^{vis}(H)$, and $HF^{\loc}(\phi^{k_j},\cl{F}^{(k_j)}) \neq 0$ for all generalized fixed points $\cl{F}$ of $\phi$ imply \begin{equation}\label{eq:newman_for_PR_II}
    \spec^{ess}(H^{(m)})=m\cdot\spec^{ess}(H) + \rho\cdot\Z.
\end{equation}

As a consequence of equations (\ref{eq:newman_for_PR_I}) and (\ref{eq:newman_for_PR_II}), and the definition of $m$ we have that $\spec^{ess}(H^{(m)})=\rho\cdot\Z + c,$ for a suitable constant $c \in \R.$ Hence, $l(F,\Gamma_{\theta})=0$ for any Hamiltonian $F$ generating $\phi^m$ and $\theta\in\rS^{1}_{\rho}$. If $\gamma(F)<\rho$, then by Lemma \ref{lma:fundamental_length} $\gamma(F)=l(F,\Gamma_{F})=0$, which is absurd since this would imply $\phi^m=\id$. Hence $\gamma(\phi^m)\geq\rho$. 
\end{proof}

\section{Hamiltonian torsion}\label{subsec: torsion intro}

\subsection{Proof of Theorem \ref{thm: torsion is perfect}}
Let $(M,\om)$ be a closed symplectic manifold and consider a non-trivial $\phi\in\Ham(M,\om)$ such that $\phi^p=\id$, where $p$ is prime. We can construct a Riemannian metric $\brat{\cdot,\cdot}$ invariant under the action of \[G=\{\id,\phi,\dots,\phi^{p-1}\},\] a fact that is true for any compact Lie group $G$. In other words, $\phi$ is an isometry with respect to this metric. We first show that $\text{Fix}(\phi)$ is composed of finitely many isolated path-connected components. 

Let $x\in\cF\subset\text{Fix}(\phi)$, where $\cF$ is the path-connected component of $x$. We claim that there exists a neighbourhood of $x$ which does not intersect any other connected component of $\text{Fix}(\phi)$. Suppose the contrary: then $x$ would be a limit point of $\text{Fix}(\phi)\setminus\cF$. In particular if $B_{\eps}(x)$ is a normal ball of radius $\eps$ around $x$, then there exists a point $y\in B_{\eps}(x)\bigcap(\text{Fix}(\phi)\setminus\cF)$ and we can consider the unique minimizing geodesic $\ga$, given by the exponential map, satisfying $\ga(0)=x$ and $\ga(1)=y$. However, $\phi$ is an isometry so we have that $\til\ga=\phi\circ\ga$ is also a minimizing geodesic satisfying $\til\ga(0)=x$ and $\til\ga(1)=y$, hence by uniqueness we must have $\text{Image}(\ga)\subset\text{Fix}(\cF)$ contradicting the fact that $y$ was in a distinct path-connected component. Since $\cF$ is compact we can choose the radius $\eps$ of the normal ball uniformly so that $\cF$ is in fact isolated, which by the compactness of $M$ implies that there are only finitely many path-connected components. 


Furthermore, if $k$ is not divisible by $p$ then we have $\text{Fix}(\phi^{k}) = \text{Fix}(\phi).$ In fact, since $p$ and $k$ are relatively prime there exist integers $a_k, b_k$ such that $a_k k + b_kp=1$. Thus,
\begin{equation*}
\phi = \phi^{a_k k + b_k p} = \phi^{a_k k}\phi^{b_k p}=\phi^{a_k k}.
\end{equation*}
So if $x$ is a fixed point of $\phi^k$ then the above equality shows that $x$ is also a fixed point of $\phi.$ Conversely, if $x$ is a fixed point of $\phi$ it is clearly a fixed point for any of its iterations. Finally, the same argument shows that if $x$ is contractible as a fixed point of $\phi^k$ it is also contractible as a fixed point of $\phi,$ and vice versa. Therefore $\fix(\phi^k) = \fix(\phi).$

To show that $\phi$ is weakly non-degenerate we utilize the fact if $M$ is connected and $f\in\text{Iso}(M,\brat{\cdot,\cdot})$ is such that $f(p)=p$ and $D(f)_p=\id_{T_pM}$ for a point $p\in M,$ then $f=\id_M$. This can be proven by considering the non-empty closed set
\begin{equation*}
S=\{q\in M \,|\, f(q)=q,\, D(f)_q=\id_{T_qM}\},
\end{equation*}
and noting that the existence of normal balls implies that $S$ is also open. Applied to our context, we must then show that for every $x\in\text{Fix}(\phi)$, $D(\phi)_x$ must have at least one eigenvalue different from 1, otherwise $\phi$ would have to be trivial. One way to see this is by noting that as $D(\phi)_x\in Sp_{2n}(T_xM)$ is an element of finite order, its Jordan form is diagonal, hence it is trivial if and only if all its eigenvalues are equal to 1. 

A slight modification of the above arguments, which amounts to the Slice theorem (see \cite[Theorem I.2.1]{Audin-TorusBook}) shows first that each path-connected component $\cl{F}$ of the fixed point set of $\phi$ is a closed connected submanifold of $M.$ Moreover, for each $\cl{F}$ and $x \in \cl{F},$ $\ker(D(\phi)_x - \id_{T_xM}) = T_x \cl{F},$ which is to say that the graph of $\phi$ intersects the diagonal $\Delta \subset M \times M^{-}$ {\em cleanly}. In other words, $\phi$ is a Floer-Morse-Bott Hamiltonian diffeomorphism. 

Finally, to prove that for a generalized fixed point $\cl{F}$ of $\phi,$ and capping $\ol{\mf{F}}$ of its corresponding generalized periodic orbit $\mf{F},$ the mean index $\Delta(H,\ol{x})$ is constant as a function of $x \in \cl{F},$ we argue as follows. We shall prove that for a fixed $x_0 \in \cl{F},$ the function $f:\cl{F} \to \R,$ given by $f(x) = \Delta(H,\ol{x}) - \Delta(H,\ol{x}_0)$ has integer values. By continuity of the mean-index this implies that $f$ is identically constant, and as $f(x_0) = 0,$ it is identically zero. This shows the required statement.

First we prove that $f$ has integer values. Similarly to the case of a Riemannian metric, by \cite[Proposition 2.5.6]{McDuffSalamonIntro3} we can find an $\om$-compatible almost complex structure $J$ on $M$ that is preserved by $\phi.$ This allows us to consider $D(\phi)_x \in Sp_{2n}(T_xM)$ for all $x \in \cl{F}$ a unitary matrix, which has diagonal Jordan form, and is determined up to conjugation by its spectrum with geometric multiplicities. Furthermore its spectrum lies in the finite set $\mu_p \subset \C$ of $p$-th roots of unity. Therefore by continuity of the spectrum in the operator norm, which holds for normal and hence for unitary matrices in particular, this implies that in fact the spectrum of $D(\phi)_x$ does not depend on $x \in \cl{F},$ and that all $D (\phi)_x$ for $x \in \cl{F}$ are conjugate by appropriate unitary isomorphisms. Therefore $D(\phi)_x$ and $D(\phi)_{x_0}$ can be connected to the identity by conjugate paths, which therefore have equal mean-indices. Now, as the paths obtained from $D(\phi^t_H)_x,$ $D(\phi^t_H)_{x_0}$ by means of the cappings differ from these conjugate paths by suitable loops $\Phi,$ $\Phi_0$ in the symplectic group, we obtain that $f(x) = \Delta(H,\ol{x}) - \Delta(H,\ol{x}_0) = \Delta(\Phi) - \Delta(\Phi_0) \in \Z.$ 

Finally, observe that $D(\phi)_x$ being $(\om_x, J_x)$-unitary, $T_x \cl{F}$ is $J_x$-invariant, and the tangent space $T_x M$ splits as a symplectic direct sum $T_x \cl{F} \oplus \mathcal{N}_x,$ where $\mathcal{N}_x$ is the normal bundle to $\cl{F}$ at $x$ (in fact this splitting can be obtained by taking $\cl{N}_x$ to be the Hermitian orthogonal complement of $T_x \cl{F}$). In particular, $\cl{F}$ is a symplectic submanifold of $(M,\om).$

Hence, the above discussion shows that $\phi$ is generalized perfect with sequence $k_j$ being the monotone-increasing ordering of any infinite subset of $\{ k \neq 0\; (\mrm{mod} \;p)\}.$

\qed

\begin{rmk}
	We have just seen that a $p$-torsion Hamiltonian diffeomorphism $\phi$ is weakly non-degenerate and generalized perfect. In our setting it is enough to consider the case $\phi$ has prime order. In fact, if $\phi$ has order $d\geq 2$ and $p$ is a prime that divides $d$, i.e. there is an integer $m$ such that $d=pm$, we consider the Hamiltonian diffeomorphism $\psi=\phi^m$ which, in turn, has prime order. Equivalently, by Cauchy's theorem, if $G$ is a finite group then for every prime $p$ dividing its order there exists an element of order $p$.
\end{rmk}

\subsection{Proof of Proposition \ref{prop: Pozniak}}

We first observe that by the universal coefficients formula, it is sufficient to prove the statement for $R = \Z.$

Now from \cite[Chapter 9, and Proof of Theorem 2.3.2]{Schmaschke} and the translation \cite[Theorem 5.2.2]{BPS-propagation}, by means of the graph construction, of \cite[Theorem 3.4.11]{Pozniak} from the setting of Lagrangian clean intersections to the Floer-Morse-Bott Hamiltonian setting, it is direct to see that there is an isomorphism \[ HF^{\loc}(\phi,\cl{F}) \cong H(\cl{F};\cl{O} \times_{\pm 1} \Z)\] of the local homology and the homology of $\cl{F}$ with coefficients in a $\Z$-local system $\cl{O} \times_{\pm 1} \Z,$ with structure group $\{\pm 1\} \cong\Z/2\Z,$ associated to a double cover $\cl{O}$ of $\cl{F}$ that we describe below. It is the goal of the proof to show that {\em in our case} this local system is trivial.

The local system $\cl{O}$ is defined as follows. For $x,y \in \cl{F},$ consider the space $\sc{P}_{x,y}(\cl{F})$ of smooth maps $\gamma:\R \to \cl{F},$ $\displaystyle\lim_{s \to -\infty}\gamma(s) = x,$ $\displaystyle\lim_{s \to \infty}\gamma(s) = y$ for which the convergence is exponential with derivatives. Let $u_{\gamma}: \R \times \rS^1 \to M$ denote the cylinder $u_{\gamma}(s,t) = \phi^t_H(\gamma(s)).$ Look at the bundle $E_{\gamma} \to \R \times \rS^1$ given by $E_{\gamma} = (u_{\gamma})^* TM.$ Now for a sufficiently small positive number $\eps > 0$ depending only on $H$ and $\cl{F}$ consider real Cauchy-Riemann differential operators \[D_{\gamma}: W^{1,p,\eps}(E_{\gamma}) \to L^{p,\eps}(E_{\gamma})\] between Sobolev spaces of sections of $E_{\gamma}$ with $\epsilon$-exponential decay as $|s| \to \infty,$ that over $(-\infty,-C)$ and $(C,\infty),$ for a large $C >0,$ coincide with {real Cauchy-Riemann operators} determined by a choice of an $\om$-compatible almost complex structure $\{J_t\} \in \cl{J}_M$ and connections whose parallel transport over the curve $\{(s,t)\}_{t \in [0,1]}$ (with $s$ fixed) is determined by the linearization of $\phi^t_H$ at $\gamma(s).$ For $\eps > 0$ sufficiently small, all these operators are Fredholm. Moreover, the auxiliary data of connections and complex structures forming a contractible space, all these operators are furthermore homotopic to each other in the space of Fredholm operators. It is shown in \cite{Schmaschke} and \cite[Chapter 8]{FO3:book-vol12} that for $\gamma,\gamma' \in \sc{P}_{x,y}(\cl{F})$ the orientation torsors $|D_{\gamma}|,$ $|D_{\gamma'}|$ of the determinant spaces $\det(D_{\gamma}),$ $\det(D_{\gamma'})$ are canonically isomorphic\footnote{Recall that the determinant line of a Fredholm operator $D$ is the real vector space of dimension $1$ defined as $\det(D) = \det(\mrm{coker}(D))^{\vee} \otimes \det(\ker(D)),$ where for a real finite-dimensional vector space $V$ of dimension $d,$ $\det(V) = \Lambda^d(V),$ and for a real vector space $l$ of dimension $1$ its orientation torsor over the group ${\pm 1}$ is $|l| = (l\setminus \{0\})/(\R_{>0}).$}. We can therefore fix $x \in \cl{F},$ and set our local system $\cl{O} \to \cl{F}$ to be induced from the sets $|D_{\gamma}|$ for $\gamma \in \sc{P}_{x,y}(\cl{F})$ for $y \in \cl{F}$ with the natural identifications provided by this isomorphism.  

Now we prove that $\cl{O}$ is trivial in our case. Suppose $\gamma \in \sc{P}_{x,y}(\cl{F}).$ It is sufficient to show that $\det(D_{\gamma})$ is canonically oriented. Now, as in the proof of Theorem \ref{thm: torsion is perfect}, in our case there exists an $\om$-compatible almost complex structure $J$ on $M$ invariant under $\phi.$ In particular $D(\phi)_x: T_x M \to T_x M$ is $(J_x,\om_x)$-unitary for all $x \in \cl{F}.$ This, together with the fact that the universal cover $\til{Sp}(2n,\R)$ deformation-retracts to the universal cover $\til{U}(n)$ of its unitary subgroup, implies that $D_{\gamma}$ is homotopic in the space of Fredholm operators, canonically up to a contractible choice of auxiliary data, to a real Cauchy-Riemann operator $D: W^{1,p,\eps}(E_{\gamma}) \to L^{p,\eps}(E_{\gamma})$ corresponding to a $J$-unitary connection. (In fact we apply a homotopy depending smoothly on $x_0 \in \cl{F},$ and from the symplectic connections on $x^*(TM) \to \rS^1$ for $x(t) = \phi^t_H(x_0),$ given by the linearlized flow of $\phi^t_H$ to unitary connections, while at all times preserving their monodromies $D(\phi^1_H)_{x_0}$ over $\rS^1$ for all $x_0 \in \cl{F}.$) But such operators are in fact complex Cauchy-Riemann operators, their kernels and cokernels are complex vector spaces, and hence their determinants are canonically oriented. Hence $|D|$ and $|D_{\gamma}|$ admit canonical elements $o,o_{\gamma}$. By a similar argument, following the definition of the isomorphisms $\psi_{\gamma,\gamma'}: |D_{\gamma}| \xrightarrow{\sim} |D_{\gamma'}|$ from \cite{Schmaschke} and \cite[Chapter 8]{FO3:book-vol12}, the key point being that orientation gluing is natural with respect to homotopies \cite[Lemma 9.4.1]{Schmaschke}, we see that $\psi_{\gamma,\gamma'}(o_{\gamma}) = o_{\gamma'}.$ Therefore $\cl{O}$ admits a continuous section, and is hence trivial. This finishes the proof. \qed

\subsection{Proof of Theorem \ref{thm: lower bound}}
Consider $\phi\in\Ham(M,\om) \setminus \{\id \}$ such that $\phi^d=\id$ and let $H$ be a Hamiltonian function generating $\phi$. Then $\gamma(H)>0$ by the non-degeneracy of the spectral norm. Since $\phi$ has finite order $d$ we have that $\{\phi_{H^{(d)}}^t\}_{t \in [0,1]}$ is a Hamiltonian loop, which, in addition to the fact that $M$ has rationality constant $\rho>0$, implies that $\Spec(H^{(d)}) = a +\rho\cdot\Z $ for a real constant $a$. One can show by a quick calculation that $\Spec(\ol{H}^{(d)}) = -a +\rho\cdot\Z.$

Furthermore, observe that $\Spec(H)\subset\Spec(H^{(d)})/d$. In fact, if $c\in\Spec(H)$ then there exists a 1-periodic capped orbit $\ol{x}\in\til\cO(H)$ such that $\cA_H(\ol{x})=c$. Consequently, $\cA_{H^{(d)}}(\ol{x}^{(d)}) = d\cdot\cA_H(\ol{x}) = d\cdot c$, which implies the claim when added to the fact that $\ol{x}^d$ is a critical point of $\cA_{H^{(d)}}$.

Finally, the above observations imply that $\gamma(H)\in(\rho/d)\cdot\Z$. In particular, the fact that it is positive implies that $\gamma(H)\geq\rho/d$. Since $H$ was an arbitrary Hamiltonian generating $\phi$ it is clear that $\gamma(\phi)\geq\rho/d$. 
\qed

\subsection{Proof of Theorem \ref{thm: Polterovich}}\label{subsec: theorem of Polterovich}
Similarly to the case of Theorem \ref{thm: lower bound}, $\phi^d = \id$ implies in the symplectically aspherical setting that for a Hamiltonian $H$ generating $\phi,$ we have $\spec(H^{(d)}) = \{a\},$ and $\spec(\ol{H}^{(d)}) = \{-a\},$ for a constant $a \in \R,$ and hence $\spec(H) \subset \spec(H^{(d)})/d = \{a/d\}$ consists of at most one point. Since $\spec(H)$ contains $c([M],H),$ we obtain that $c([M],H) = a/d.$ Similarly, $c([M],\ol{H}) = -a/d.$ This means that $\gamma(H) = 0$ and hence $\gamma(\phi) = 0,$ which implies by non-degeneracy of $\gamma$ that $\phi = \id.$ This finishes the proof. \qed

\subsection{Proof of Theorem \ref{thm: spectral Newman}}
Consider $\phi \in\Ham(M,\om) \setminus \{\id\}$ such that $\phi^p=\id$ for a prime number $p$. Fix a coefficient field $\bK.$ We show that there exists a positive integer $m$, such that
\begin{equation}\label{eq:newman}
    \gamma(\phi^m)\geq \frac{\lfloor\frac{p}{2}\rfloor}{p}\cdot\rho,
\end{equation}
where $\lfloor\frac{p}{2}\rfloor$ denotes the floor of $p/2$. We may suppose $p\geq3$, since the case $p=2$ is settled by Theorem \ref{thm: lower bound} for $m=1$. In this case, note that $\lfloor\frac{p}{2}\rfloor = \frac{p-1}{2}$. Supposing that $\gamma(\phi)< \rho(p-1)/2p$, we can find a Hamiltonian $H$ generating $\phi$ satisfying $\gamma(H)< \rho(p-1)/2p$. In the proof of Theorem \ref{thm: lower bound} we saw that $\gamma(H)$ must be a positive integer multiple of $\rho/p$. Therefore, by Lemma \ref{lma:fundamental_length} we can find a positive integer $r\leq (p-3)/2$, such that 
\begin{equation}\label{eq:newman_I}
l(H,\Gamma_{H})=\gamma(H)=\frac{r\rho}{p}.     
\end{equation}
In particular, we have that $2r<p$, which in addition to the fact that $p>2$, implies that there exist integers $a,b$ such that $a(2r)+bp=1$. Observe that $b$ must be an odd integer since $a(2r)$ is even while $p$ and $1$ are odd. Let $k$ be the integer such that $-b=2k+1$. Furthermore, note $a\neq0$, and set $m=|a|$. There are two cases to be considered depending on the sign of the integer $a$. 

If $a>0$ we have that $m(2r)-(2k+1)p=1$, which implies
\begin{equation}\label{eq:newman_II}
    \frac{mr}{p}-\frac{p+1}{2p}=k,
\end{equation}
where $(p+1)/2=\lceil p/2\rceil$. Furthermore, since $m$ and $p$ are coprime, Theorem \ref{thm: torsion is PR} implies that 
\begin{equation}\label{eq:newman_III}
    \spec^{ess}(H^{(m)}) = m\cdot\spec^{ess}(H) + \rho\Z.
\end{equation}
Combining (\ref{eq:newman_I}), (\ref{eq:newman_II}) and (\ref{eq:newman_III}) we obtain that there exist $c_0,c_1\in\spec^{ess}(H)$ such that 
\begin{equation*}
    mc_1-mc_0=\frac{mr\rho}{p}=\frac{(p+1)\rho}{2p} + k\rho.
\end{equation*}
In addition, $mc_1+j\rho$ and $mc_0+j\rho$ belong to the essential spectrum of $H^{(m)}$ for every integer $j$. We conclude that for each $\theta\in\rS^{1}_{\rho}$ and $\text{I}\in\Gamma_{\theta}$, there exists an integer $l$ such that either
\begin{equation*}
    mc_1+l\rho,mc_0+(k+l)\rho\in\text{I},
\end{equation*}
or
\begin{equation*}
    mc_1+l\rho,mc_0+(k+(l+1))\rho\in\text{I}.
\end{equation*}
Consequently,
\begin{align*}
    l(H^{(m)},\Gamma_{\theta})&\geq\min\{mc_1-mc_0-k\rho,\,mc_0- mc_1+(k+1)\rho\}\\
        &=\min\bigg\{\frac{(p+1)\rho}{2p},\, \frac{(p-1)\rho}{2p}\bigg\}=\frac{(p-1)\rho}{2p}.
\end{align*}
Since, $\theta$ was arbitrary we conclude 
\begin{equation}\label{eq:newman_IV}
    l(\phi^m)\geq \frac{\lfloor\frac{p}{2}\rfloor}{p}\cdot\rho.
\end{equation}
When $a<0$, an analogous argument can be made to show that once again (\ref{eq:newman_IV}) is obtained. Hence, by Lemma \ref{lma:spectral_length} we obtain inequality (\ref{eq:newman}). \qed

\subsection{Proof of Theorem \ref{thm: PR}}
Consider a generalized pseudo-rotation $\phi$ as in Lemma \ref{lma:newman_for_PR}. As consequence of this lemma, we may suppose that there exist $c_1,c_2\in\spec^{ess}(H)$ such that $c_1-c_2\in\rho \cdot (\R\setminus\Q)$, otherwise $\gamma(\phi^{m})\geq\rho$ for some positive integer $m$. Since the orbit of any irrational rotation in $\rS^1$ is dense, for every $\eps>0$ there exists an integer $m_{\eps}$ such that
\begin{equation*}
    \frac{\rho}{2}-\eps<d_{S^1_{\rho}}([c_1],[m_{\eps}\cdot c_2]) \leq\frac{\rho}{2},
\end{equation*}
where for $x \in \R,$ we denote by $[x] \in S^1_{\rho} = \R/\rho\Z$ its equivalence class, and $d_{S^1_{\rho}}$ is the distance function on $S^1_{\rho}$ coming from the standard flat metric on $\R.$ Therefore, arguing as in the proof of Theorem \ref{thm: spectral Newman} we conclude
\begin{equation*}
    \sup_{k\in\Z_{>0}}\gamma(\phi^{k})\geq\frac{\rho}{2}.
\end{equation*}
\qed

\section{Proofs of no-torsion theorems}



The proofs of Theorems \ref{thm: generalized perfect CY negmon} and \ref{thm: Morse-Bott PR uniruled} rely on the following observations regarding the mean-index. First, let $\til{\phi}$ be a lift of $\phi$ to the universal cover $\til{\Ham}(M,\om)$ of $\Ham(M,\om).$ As our path-connected isolated fixed point sets are weakly non-degenerate, if the capping $\ol{\mf F}$ of the generalized $1$-periodic orbit $\mf F$ corresponding to an isolated fixed point set $\cl F \subset \fix(\phi)$ carries a cohomology class $\mu$ of Conley-Zehnder index $n$ in $HF^n(\til{\phi}) \cong QH^{2n}(M,\Lambda_{\bK}),$ for a coefficient field $\bK,$ then its mean-index $\Delta = \Delta(\til{\phi},\ol{\mf F})$ satisfies $\Delta - n <n< \Delta + n.$ Hence \begin{equation}\label{eq: Delta pos} \Delta(\til{\phi},\ol{\mf F}) \in (0,2n).\end{equation} 

Similarly, if $\ol{\mf F}$ carries a homology class $u \in HF_n(\til{\phi}) \cong QH_{2n}(M,\Lambda_{\bK}),$ then \eqref{eq: Delta pos} holds. Both of these implications follow from Lemma \ref{lma:carrier_spectral_invariant_genralized-QH}, Equation \eqref{eq:supp_local_FH_generalized}, and Section \ref{subsec: Floer coh}. We will specifically use the case $u = [M],$ which follows from Lemma \ref{lma:carrier_spectral_invariant_genralized}.

\subsection{Proof of Theorem \ref{thm: generalized perfect CY negmon}}
We first treat the negative monotone case. Choose $H \in \cl H,$ such that the path $\{\phi^t_H \}_{t \in [0,1]}$ represents the class $\til{\phi}$ lifting $\phi.$ Let $k_i$ be the sequence associated to $\phi$ as a generalized perfect diffeomorphism. By the pigeonhole principle, there exists an isolated fixed point set $\cl F \subset \fix(\phi),$ and an increasing subsequence of $k_i$ that we renumber and denote by $r_i,$ such that $c([M], H^{(r_i)})$ is carried by a capping $\ol{\mf{G}}_i$ of the isolated set of $1$-periodic orbits of the $r_i$-iterated Hamiltonian $H^{(r_i)}$ corresponding to ${\cl{F}}^{(r_i)}.$ Set $\ol{\mf{G}}= \ol{\mf{G}}_1.$ If $r_1$ divides all $r_i,$ by taking a power of $\phi,$ we can assume that $r_1 = 1.$ We proceed with the proof in this case, and then explain how to modify it in the general case.

Write $\ol{\mf{G}}_i$ as a recapped iteration of $\ol{\mf{G}},$ that is \begin{equation}\label{eq: recap 0} \ol{\mf{G}}_i = \ol{\mf{G}}^{(r_i)} \# A_i.\end{equation} We claim that for $r_i$ large, $\om(A_i) \geq 0,$ and $c_1(A_i) > 0$ contradicting negative monotonicity. Indeed, write $\cA_i$ for the action functional of $H^{(r_i)},$ and $\cA:=\cA_1.$ Then by \eqref{eq: recap 0} and the triangle inequality for spectral invariants, \begin{equation}\label{eq: large omega} r_i \cA(\ol{\mf{G}}) - \om(A_i) = \cA_i(\ol{\mf{G}}_i) = c([M], H^{(r_i)}) \leq r_i c([M], H) = r_i \cA(\ol{\mf{G}}).\end{equation} Hence \[\om(A_i) \geq 0.\] However, as $\ol{\mf{G}}_i$ carries $c([M], H^{(r_i)}),$ by \eqref{eq: Delta pos} we have $\Delta(H^{(r_i)},\ol{\mf{G}}_i) \in (0,2n)$ and also $\Delta(H,\ol{\mf{G}}) \in (0,2n).$ Hence $r_i \Delta(H, \ol{\mf{G}}) > 2n$ for $r_i$ large enough, and \begin{equation}\label{eq: large mean-index} 2n > \Delta(H^{(r_i)}, \ol{\mf{G}}_i) = r_i \Delta(H, \ol{\mf{G}}) - 2c_1(A_i).\end{equation} Therefore \[c_1(A_i) > 0,\] which finishes the proof.

If $r_1$ does not divide all $r_i,$ let $d = \gcd(\{r_i\}_{i\geq 1}).$ Then $\cl F \subset \fix(\phi^d).$ Passing to $\phi^d$ instead of $\phi,$ and to $r_i/d$ instead of $r_i,$ we can and will assume that $d = 1.$ Now take $\ol{\mf{H}}$ to be a capping of $\mf{F}$ for $H,$ and write $\ol{\mf{G}}_i = \ol{\mf{H}}^{(r_i)} \# B_i.$ Now extend $\om$ and $c_1$ to negatively proportional morphisms $\Gamma^{\Q} = \Gamma \otimes_{\Z} \Q \to \R.$ For each $i \geq 1,$ set $A_i = B_i - \frac{r_i}{r_1} B_1 \in \Gamma^{\Q}.$ Now setting $\ol{\mf{G}} = \ol{\mf{G}}_{1},$ the estimates \eqref{eq: large omega}, \eqref{eq: large mean-index} hold still, which yields $\om(A_i) \geq 0,$ $c_1(A_i) > 0,$ contradicting negative monotonicity.


We now prove the symplectic Calabi-Yau case of the theorem. In this case, the mean-index of each capped fixed point set $\ol{\mf F}$ does not depend on the capping. Hence we write $\Delta(H,\cl F)$ for each generalized fixed point $\cl F$ for this mean index. Then for each positive sequence $k_i \to \infty$ of iterations with $\phi^{k_i}$ having a fixed finite number of weakly non-degenerate generalized fixed points, we argue as follows. For each $\cl F \in \pi_0(\fix(\phi)),$ \[\Delta(H^{(k_i)}, \cl F^{(k_i)}) = (k_i/k_1) \Delta(H^{(k_1)}, \cl F^{(k_1)}).\] Hence, if $\Delta(H^{(k_1)}, \cl F^{(k_1)}) > 0$ then $\Delta(H^{(k_i)}, \cl F^{(k_i)}) > 2n$ for all $k_i$ sufficiently large, and if $\Delta(H^{(k_1)}, \cl F^{(k_1)}) \leq 0$ then $\Delta(H^{(k_i)}, \cl F^{(k_i)}) \leq 0$ for all $k_i.$ Now, as each $\cl F$ is weakly non-degenerate, we obtain by the support property of local Floer homology that in both cases $HF_{n}^{\loc}(H^{(k_i)},\cl F^{(k_i)}) = 0.$ Moreover, the following stronger statement is true: for all $k_i$ sufficiently large, $H^{(k_i)}$ admits a $C^2$-small non-degenerate Hamiltonian perturbation $H_i$ without capped periodic orbits of Conley-Zehnder index $n.$ However, this is in contradiction to the existence of the PSS isomorphism. Specifically, in this case $HF_n(H_i) = 0$ by definition of Floer homology, and by the PSS isomorphism $HF_n(H_i) \cong QH_{2n}(M) \neq 0.$ Indeed $[M] \in QH_{2n}(M)$ is non-zero. \qed

The following result was first proven in \cite{S-PRQS} in the setting of a {\em pseudo-rotation} assuming that the quantum Steenrod square of the point cohomology class is undeformed, or in other words that $(M,\om)$ is not $\F_2$ Steenrod uniruled. We observe that the same statement holds for generalized pseudo-rotations, with essentially the same proof, and with a small modification following \cite{SSW}, for all primes $p.$ Here $\mu \in QH^{2n}(M,\Lambda_{\F_p})$ denotes the cohomology class Poincar\'{e} dual to the point.

%
\begin{thm}\label{thm: spec ineq}
	Let $\psi$ be a generalized $\bF_p$ {\pr} with sequence $k_j = p^{j-1}$ of a closed monotone symplectic manifold $(M,\om)$ that is not $\F_p$ Steenrod uniruled. Then \begin{equation}\label{eq: greater or equal} c(\mu,\til{\psi}^p) \geq p \cdot c(\mu, \til{\psi})\end{equation} for each $\til{\psi} \in \til{\Ham}(M,\om)$ covering $\psi.$
\end{thm}

We proceed to the proof of Theorem \ref{thm: Morse-Bott PR uniruled}.

\subsection{Proof of Theorem \ref{thm: Morse-Bott PR uniruled}}
Choose $H \in \cl H,$ such that the path $\{\phi^t_H \}_{t \in [0,1]}$ represents the class $\til{\phi}$ lifting $\phi.$ By the pigeonhole principle, there exists an isolated fixed point set $\cl F \subset \fix(\phi),$ and an increasing sequence $k_i$ such that $c(\mu, H^{(r_i)})$ for $r_i =  p^{k_i}$ is carried by a capping $\ol{\mf{G}}_i$ of the isolated set of $1$-periodic orbits of the $r_i$-iterated Hamiltonian $H^{(r_i)}$ corresponding to ${\cl{F}}^{(r_i)}.$ By taking a power of $\phi,$ we can assume that $r_1 = 1,$ and set $\ol{\mf{G}}= \ol{\mf{G}}_1.$ Write $\ol{\mf{G}}_i$ as a recapped iteration of $\ol{\mf{G}},$ that is \begin{equation}\label{eq: recap} \ol{\mf{G}}_i = \ol{\mf{G}}^{(r_i)} \# A_i.\end{equation} We claim that for $r_i$ large, $\om(A_i) \leq 0,$ and $c_1(A_i) > 0$ contradicting  monotonicity. Indeed, write $\cA_i$ for the action functional of $H^{(r_i)},$ and $\cA:=\cA_1.$ Then by \eqref{eq: recap} and Theorem \ref{thm: spec ineq}, \[r_i \cA(\ol{\mf{G}}) - \om(A_i) = \cA_i(\ol{\mf{G}}_i) = c(\mu, H^{(r_i)}) \geq r_i c(\mu, H) = r_i \cA(\ol{\mf{G}}).\] Hence \[\om(A_i) \leq 0.\] However, as $\ol{\mf{G}}_i$ carries $c(\mu, H^{(r_i)}),$ by \eqref{eq: Delta pos} we have $\Delta(H^{(r_i)},\ol{\mf{G}}_i) \in (0,2n)$ and also $\Delta(H,\ol{\mf{G}}) \in (0,2n).$ Hence $r_i \Delta(H,\ol{\mf{G}}) > 2n$ for $r_i$ large enough, and \[2n > \Delta(H^{(r_i)}, \ol{\mf{G}}_i) = r_i \Delta(H, \ol{\mf{G}}) - 2c_1(A_i).\] Therefore \[c_1(A_i) > 0.\] 

\subsection{Proof of Theorem \ref{thm: torsion is PR}}
	Suppose that $\phi \in \Ham(M,\om) \setminus \{ \id \}$ is of prime order $q \geq 2.$ Let $p \geq 2$ be a prime different from $q.$ In particular $\phi^{j\cdot p^k} \neq \id$ for all $k \in \Z,$ $1 \leq j \leq q-1.$ 
	
	
	
	Denote \[B(\phi,\F_p) = \max_{1 \leq j\leq q-1} {\beta_{\mrm{tot}}(\phi^j,\F_p)}.\] By Theorem \ref{thm: Smith} we obtain for $1 \leq j \leq q-1$ that \[B(\phi,\F_p) \geq \beta_{\mrm {tot}}(\phi^{j \cdot p^k},\F_p) \geq p^{k}\beta_{\mrm {tot}}(\phi^j,\F_p).\] Choosing a sufficiently large positive $k,$ this implies that for all $1 \leq j \leq q-1,$ \[ \beta_{\mrm{tot}}(\phi^j,\F_p) = 0,\] whence by Proposition \ref{prop: barcodes} all such $\phi^j$ are generalized $\F_p$ pseudo-rotations. They are weakly non-degenerate by Theorem \ref{thm: torsion is perfect}. In other words, the equality \[\spec^{vis}(H; \F_p) = \spec^{ess}(H; \F_p)\] follows directly from the fact that $\beta_{\mrm{tot}}(\phi, \F_p) = 0.$ This finishes the proof of part \ref{G: case 1}.
	
	
	Let us prove that $\spec^{vis}(H^{(k)}; \Q) = k\cdot \spec^{vis}(H; \Q) + \rho\cdot \Z$ for all $k \in \Z$ coprime with $q.$ By the universal coefficient formula in local Floer homology, it is sufficient to prove the identity $\spec^{vis}(H^{(k)}; \F_p) = k\cdot \spec^{vis}(H; \F_p) + \rho\cdot \Z$ for coefficients in $\F_p$ for an infinite sequence of primes $p.$ Consider the primes $p$ for which $p = a\; (\mrm{mod}\;q)$ where $a \in (\F_q)^{\ast}$ is a cyclic generator of the multiplicative group $(\F_q)^{\ast} = GL(1,\F_q)$ of $\F_q.$ In this case the set $\{\phi^{p^{j}}\;|\; j \in \Z_{\geq 0}\}$ coincides with $\{\phi^k\;|\; 1 \leq k \leq q-1\} = \{\phi^k\;|\;  k \neq 0 \; (\mrm{mod}\;q)\}.$ 
	
	Let $\ol{\mf{F}}$ be a capped generalized periodic orbit of $H.$ It is enough to prove that \[\dim_{\F_p} HF^{\loc}(H^{(p^j)},\ol{\mf{F}}^{(p^j)}) = \dim_{\F_p} HF^{\loc}(H,\ol{\mf{F}})\] for all $j \in \Z_{\geq 0}.$ Indeed, as explained above, each capped generalized fixed point of $\phi^{k}$ is a recapping of a $p^j$-iterated capped generalized fixed point of $\phi.$ 
	
	Now we know by the Smith inequality in generalized local Floer homology, Proposition \ref{thm: Smith loc}, that $\dim_{\F_p} HF^{\loc}(H^{(p^j)},\ol{\mf{F}}^{(p^j)})$ is an increasing function of $j.$ However, by the finite order condition it takes only a finite number of values. Therefore it must be identically constant. This finishes the proof of part \ref{G: case 2}.
	
	Now we prove part \ref{G: case 3} relying on Proposition \ref{prop: Pozniak}. First let $p=q.$ Then for $\psi = \phi^k,$ with $k$ coprime to $p,$ \[N(\psi,\F_p) = \sum \dim_{\F_p} HF^{\loc}(\psi,\cl{F}) = \sum \dim_{\F_p} H(\cl{F};\F_p),\] the sum running over all contractible generalized fixed points $\cl{F}$ of $\psi,$ since by Proposition \ref{prop: Pozniak}  \[HF^{\loc}(\psi,\cl{F}) \cong H(\cl{F};\F_p)\] for all generalized fixed points $\cl{F}.$ We remark that $H(\cl{F};\F_p) \neq 0.$ Now, by Proposition \ref{prop: barcodes}, we know that \[N(\psi,\F_p) \geq \dim_{\F_p} H(M;\F_p).\] On the other hand by the classical Smith inequality \cite{Smith-original,Floyd-original,Borel-Transformation} we have \[ \sum \dim_{\F_p} H(\cl{F};\F_p) \leq \dim_{\F_p} H(M;\F_p),\] the sum running over all the generalized fixed points of $\psi.$ This yields \[N(\psi,\F_p) = \dim_{\F_p} H(M;\F_p).\] This finishes the proof of the first statement of part \ref{G: case 3}.
	
	\begin{rmk}
	Note that, surprisingly enough, this argument also proves that all the generalized fixed points of $\psi$ are contractible. Indeed the upper bound holds for {\em all} the generalized fixed points of $\psi,$ and if $\psi$ had a non-contractible generalized fixed point, then by Proposition \ref{prop: Pozniak} it would contribute $\dim_{\F_p} HF^{\loc}(\psi,\cl{F}) = \dim_{\F_p} H(\cl{F};\F_p) > 0$ to the sum, making the equality impossible. Alternatively, one can argue by means of a suitable generalization of Theorem \ref{thm: Smith} with $p \neq q.$
	\end{rmk}

 To prove the second statement of part \ref{G: case 3}, we first note that for $\mrm{char}(\bK) = p,$ \[\Spec^{ess}(H,\bK) = \Spec^{ess}(H,\F_p),\;\Spec^{vis}(H,\bK) = \Spec^{vis}(H,\F_p),\] by Proposition \ref{prop: barcodes}, and the equality \[\Spec^{ess}(H,\F_p) = \Spec^{vis}(H,\F_p)\] follows by the first statement of part \ref{G: case 3} and Proposition \ref{prop: barcodes}. For the last part, we note that by Proposition \ref{prop: Pozniak}, $\Spec^{vis}(H,\bK) = \Spec(H)$ because \[\dim HF^{\loc}(H,\ol{\mf{F}}) = \dim HF^{\loc}(\phi, \cl{F}) = \dim H(\cl{F}; \bK) > 0\] for all capped contractible generalized $1$-periodic orbits $\ol{\mf{F}}$ of $H.$ Now for $k$ coprime to $q,$ $\Spec(H^{(k)}) = \{ \cl{A}_{H^{(k)}}(\ol{\mf{F}}^{(k)} \# A) \}$ where the set runs over all $A \in \Gamma,$ and $\ol{\mf{F}}$ runs over all capped contractible generalized $1$-periodic orbits $\ol{\mf{F}}$ of $H.$ Indeed, all the contractible generalized fixed points of $\phi^k$ are of the form $\cl{F}^{(k)}$ for $\cl{F}$ a contractible generalized fixed point of $\phi,$ and the identity quickly follows. Now using the homogeneity and the recapping properties of the action functional we obtain  \[ \Spec(H^{(k)}) = k \cdot \Spec(H) + \rho \cdot \Z.\] Combined with the identity $\Spec^{ess}(H^{(k)}; \bK) = \Spec^{vis}(H^{(k)}; \bK) = \Spec(H^{(k)}),$ $\Spec^{ess}(H; \bK) = \Spec^{vis}(H; \bK) = \Spec(H)$ this finishes the proof. 
	
%

\section*{Acknowledgements}
We thank Leonid Polterovich for inspiring discussions. At the University of Montr\'{e}al E.S. was supported by an NSERC Discovery Grant, by the Fonds de recherche du Qu\'{e}bec - Nature et technologies, and by the Fondation Courtois, and M.S.A. was supported by the Fondation Courtois.

\bibliographystyle{abbrv}
\bibliography{bibliographyHTU}

\end{document}